\documentclass[accepted]{uai2022} %

\usepackage[dvipsnames]{xcolor}

\usepackage[american]{babel}

\usepackage{natbib} %
\usepackage{mathtools} %
\usepackage{booktabs} %
\usepackage{tikz} %

\usepackage{amsmath,amsfonts,bm}

\def\eqref#1{Equation~(\ref{#1})}

\def\1{\bm{1}}

\def\vv{{\bm{v}}}

\DeclareMathAlphabet{\mathsfit}{\encodingdefault}{\sfdefault}{m}{sl}
\SetMathAlphabet{\mathsfit}{bold}{\encodingdefault}{\sfdefault}{bx}{n}

\newcommand{\E}{\mathbb{E}}

\DeclareMathOperator*{\argmax}{arg\,max}
\DeclareMathOperator*{\argmin}{arg\,min}

\usepackage{url}

\usepackage{booktabs}       %
\usepackage{amsfonts}       %
\usepackage{nicefrac}       %
\usepackage{microtype}      %

\usepackage{geometry}
\usepackage{amsmath}
\usepackage{amssymb}

\usepackage{graphicx}
\usepackage{subfigure}
\usepackage{multirow}
\usepackage{makecell}
\usepackage{wrapfig}

\usepackage{caption}
\usepackage{wrapfig}

\usepackage{algorithm,algpseudocode}
\usepackage{qiangstyle}

\title{Pareto Navigation Gradient Descent: a First-Order Algorithm for Optimization in Pareto Set}

\author[1]{\href{mailto:<maoye21@utexas.edu>?Subject=Your UAI 2022 paper}{Mao Ye}{}}
\author[1]{Qiang Liu}
\affil[1]{%
    Computer Science Dept.\\
    The University of Texas at Austin.
}

\begin{document}

\global\long\def\th{\theta}%
\global\long\def\E{\mathrm{E}}%
\global\long\def\L{\vv{\ell}}%
\global\long\def\dge{\succeq}%
\global\long\def\dle{\preceq}%
\global\long\def\dsqe{\succ}%
\global\long\def\dsle{\prec}%
\global\long\def\C{\mathcal{C}}%
\global\long\def\P{\mathcal{P}}%
\global\long\def\df{\mathrm{d}}
\global\long\def\tim{{t}}

\newtheorem{theorem}{Theorem}
\newtheorem{lemma}{Lemma}
\newtheorem{assumption}{Assumption}

\maketitle

\begin{abstract}
Many modern machine learning applications, such as multi-task learning, require finding optimal model parameters to trade-off multiple objective functions that may conflict with each other.
The notion of the Pareto set allows us to focus on the set of (often infinite number of)
models that cannot be strictly improved. But it does not provide an actionable procedure for picking one or a few special models to return to practical users. In this paper, we consider \emph{optimization in Pareto set (OPT-in-Pareto)}, the problem of finding Pareto models that optimize an extra reference criterion function within the Pareto set. This function can either encode a specific preference from the users, or represent a generic diversity measure for obtaining a set of diversified Pareto models that are representative of the whole Pareto set.
Unfortunately, despite being a highly useful framework, efficient algorithms for OPT-in-Pareto have been largely missing, especially for large-scale, non-convex, and non-linear objectives in deep learning. A naive approach is to apply Riemannian manifold gradient descent on the Pareto set, which yields a high computational cost due to the need for eigen-calculation of Hessian matrices. We propose a first-order algorithm that approximately solves OPT-in-Pareto using only gradient information, with both high practical efficiency and theoretically guaranteed convergence property. Empirically, we demonstrate that our method works efficiently for a variety of challenging multi-task-related problems.
\end{abstract}

\section{Introduction}\label{sec:intro}
Although machine learning tasks are traditionally framed as optimizing a single objective, many modern applications, especially in areas like multitask learning, require finding optimal model parameters to minimize multiple objectives (or tasks) simultaneously. As the different objective functions may inevitably conflict with each other, the notion of optimality in multi-objective optimization (MOO) needs to be characterized by the Pareto set: the set of model parameters whose performance of all tasks cannot be jointly improved.

Focusing on the Pareto set allows us to filter out models that can be strictly improved. However, the Pareto set typically contains an infinite number of parameters that represent different trade-offs of the objectives. 
For $m$ objectives $\ell_1,\ldots, \ell_m$, 
the Pareto set is often an $(m-1)$ dimensional manifold. 
It is both intractable and unnecessary to give practical users the whole exact Pareto set. A more practical demand is to find some user-specified special parameters in the Pareto set, which can be framed into the following \emph{optimization in Pareto set (OPT-in-Pareto)} problem: 

\emph{Finding one or a set of parameters inside the Pareto set of $\ell_1,\ldots, \ell_m$ that minimize a reference criterion $F$.}

Here the criterion function $F$ can be used to encode 
an \emph{informative} 
user-specific preference on the objectives $\ell_1,\ldots, \ell_m$, which allows us to provide the best models customized for different users. $F$ can also be an \emph{non-informative} measure 
that encourages, for example, the diversity of a set of model parameters. In this case, optimizing $F$ in Pareto set gives a set of diversified Pareto models that are representative of the whole Pareto set, from which different users can pick their favorite models during the testing time. 

OPT-in-Pareto provides a highly generic and actionable framework for multi-objective learning and optimization. However, efficient algorithms for solving OPT-in-Pareto have been largely lagging behind in deep learning where the objective functions are non-convex and non-linear. Although has not been formally studied, a straightforward approach is to apply manifold gradient descent on $F$ in the Riemannian manifold formed by the Pareto set \citep{hillermeier2001generalized, bonnabel2013stochastic}. However, this casts prohibitive computational cost due to the need for eigen-computation of Hessian matrices of $\{\ell_\i\}$. In the optimization and operation research literature, there has been a body of work on OPT-in-Pareto viewing it as a special bi-level optimization problem \citep{dempe2018bilevel}. However, these works often heavily rely on the linearity and convexity assumptions and are not applicable to the non-linear and non-convex problems in deep learning; see for examples in
\citet{ecker1994optimizing,jorge2005bilinear,thach2014problems,liu2018primal,sadeghi2021solving} (just to name a few). In comparison, the OPT-in-Pareto problem seems to be much less known and under-explored in the deep learning literature. The exceptions are three works \citep{mahapatra2020multi,kamani2021pareto,chen2021weighted} that propose specialized algorithms for some specific instantiations of the OPT-in-Pareto problem and we defer a more detailed review to Section \ref{sec: review}.

In this work, we provide a practically efficient first-order algorithm for OPT-in-Pareto, using only gradient information of the criterion $F$ and objectives $\{\ell_i\}$. Our method, named \emph{Pareto navigation gradient descent}  ({\PNG}), iteratively updates the parameters following a direction that carefully balances the descent on $F$ and $\{\ell_i\}$, such that it guarantees to move towards the Pareto set of $\{\ell_i\}$ when it is far away, and optimize $F$ in a neighborhood of the Pareto set. Our method is simple, practically efficient and has theoretical guarantees.

In empirical studies, 
we demonstrate that our method 
works efficiently for both optimizing user-specific 
criteria and diversity measures. %
In particular, for finding representative Pareto solutions, 
we propose an energy distance criterion
whose minimizers distribute uniformly 
on the Pareto set asymptotically 
\citep{hardin2004discretizing}, 
yielding a principled and efficient Pareto set approximation method that compares favorably with recent works such as
\citet{lin2019pareto,mahapatra2020multi}. %
We also apply {\PNG} to improve the performance of JiGen \citep{carlucci2019domain}, a multi-task learning approach for domain generalization, by using the adversarial feature discrepancy as the criterion objective.

\section{Background on Multi-objective Optimization} 
\label{sec: background}
We introduce the background on multi-objective optimization (MOO) and Pareto optimality. 
For notation, we denote by $[m]$ the integer set $\{1,2,....,m\}$, and 
$\RRplus$  the set of non-negative real numbers. 
Let $\C^m = \left\{\omega\in \RRplus^m,~~\sum_{\i=1}^m \omega_\i = 1\right\}$ be the probability simplex.
We denote by $\left\Vert \cdot\right\Vert$ the 
Euclidean norm. 

Let $\th \in \RR^\dimcc$ be a parameter of interest (e.g., the weights in a deep neural network). 
Let $\L(\cc)=[\ell_1(\cc),\ldots, \ell_m(\cc)]$ be a set of objective functions that we want to minimize. 
For two parameters $\th,\th'\in \RR^\dimcc$,
we write $\L(\cc) \succeq \L(\cc')$ if $\ell_\i(\cc) \geq \ell_\i(\cc')$ for all $\i \in [m]$; 
and write  $\L(\cc) \succ \L(\cc')$  if  
$\L(\cc) \succeq \L(\cc')$ and $\L(\cc) \neq \L(\cc')$. 
We say that $\th$  is Pareto dominated (or Pareto improved)  by $\th'$ if  $\L(\cc) \succ \L(\cc')$.  
We say that $\cc$ is Pareto optimal on a set $\Theta\subseteq \RR^\dimcc$, denoted as $\theta\in \mathrm{Pareto}(\Theta)$, if
there exists no $\theta' \in \Theta$ such that $\L(\th)\dsqe\L(\th')$. 

The Pareto global optimal set  $\P^{**} \defeq \mathrm{Pareto}(\RR^{\dimcc})$  
is the set of points (i.e., $\cc$) which are Pareto optimal on the whole domain $\RR^\dimcc$.  
The Pareto local optimal set %
of $\L$, denoted by $\P^{*}$, 
is the set of points which are Pareto optimal on a neighborhood of itself: 
\begin{align*}
\P^{*}:=\{\th\in\RR^\dimcc: ~~
& \text{there exists a neighborhood $\mathcal N_\cc$ of $\cc$, }
\\
& \text{such that $\cc\in \mathrm{Pareto}(\mathcal N_{\cc})$} \}. 
\end{align*}
The (local or global) Pareto front 
is the set of objective vectors 
achieved  by the Pareto optimal points, e.g., %
the local Pareto front is 
$\mathcal F^* = \{\L(\th):\th\in\P^*\}$. 
Because finding global Pareto optimum 
is intractable for non-convex objectives in deep learning, 
we focus on  Pareto local optimal sets in this work; 
in the rest of the paper, terms like ``Pareto set'' and ``Pareto optimum'' refer to Pareto local optimum by default. 

\paragraph{Pareto Stationary Points} 
Similar to the case of single-objective optimization, 
Pareto local optimum implies a notion of Pareto stationarity defined as follows. Assume $\L$ is differentiable on $\RR^\dimcc$. A point $\cc$ is called Pareto stationary if there must exists a set of non-negative weights  $\omega_1,\ldots, \omega_m$ with $\sum_{\i=1}^m \omega_\i = 1$, such that $\cc$ is a stationary point of the $\omega$-weighted linear combination of the objectives: $\ell_{\omega}(\cc)\defeq \sum_{\i=1}^m \omega_\i \ell_\i(\cc).$ 
Therefore, the set of Pareto stationary points, denoted by $\P$, 
can be characterized by
\begin{align}\label{equ: pareto stationary}
    \P&:=\left\{ \th\in\Theta:g(\th)=0\right\}
    \\
    \nonumber
    g(\th)&:=\min_{\omega\in\C^m}||\sum_{\i=1}^m \omega_\i\nabla\ell_\i(\th)||^2, 
\end{align}
where $g(\cc)$ is the minimum squared gradient norm of $\ell_{\omega}$ among all $\omega$ in the probability simplex $\C^m$ on $[m]$. 
Because $g(\th)$ can be calculated in practice, 
it provides an essential way to access Pareto local optimality. Being a Pareto stationary point is a necessary condition of being a Pareto local optimum.

\paragraph{Finding  Pareto Optimal Points}  
A main focus of the MOO literature is to find a (set of) Pareto optimal points. 
The simplest approach is \emph{linear scalarization}, 
which minimizes $\ell_{\omega}$ for some weight $\omega$ (decided, e.g., by the users) in $\C^m$. 
However, %
linear scalarization can only find Pareto points that lie  on 
the \emph{convex envelop} of the Pareto front   \citep[see e.g.,][]{boyd2004convex}, and hence does not give a complete profiling of the Pareto front when the objective functions (and hence their Pareto front) are non-convex.

\emph{Multiple gradient descent (MGD)} \citep{desideri2012multiple} %
is an gradient-based algorithm that can 
converge to a Pareto local optimum that lies on either the convex or non-convex parts of the Pareto front, depending on the initialization.  %
MGD starts from some initialization $\cc_0$ and updates $\cc$ at the $\k$-th iteration by 
\begin{align}
\label{equ: update mgd} 
\cc_{\k+1} & \gets \cc_\k - \xi v_\k,
\\
\nonumber
v_\k & \defeq 
\argmax_{v\in \RR^\dimcc}
\left\{ 
 \min_{\i\in[m]} \dd \ell_\i(\cc_\k)\tt v 
-\frac{1}{2}\norm{v} ^{2}
\right\},
\end{align}
where $\xi$ is the step size  and 
$v_\k$ is 
an 
update direction that maximizes 
the \emph{worst} descent rate among all objectives, since 
$%
 \nabla\ell_\i(\th_\k)\tt v  \approx 
 (\ell_\i(\theta_\k) - \ell_\i(\theta_\k-\xi v))/\xi$
 approximates the descent rate of objective $\ell_\i$ 
 when following direction $v$.  
When using a sufficiently small step size $\xi$, MGD ensures to yield a \emph{Pareto improvement} (i.e, decreasing all the objectives) on $\cc_\k$ unless $\cc_\k$ is Pareto (local) optimal; this is because the optimization in \eqref{equ: update mgd} always yields $\min_{\i\in[m]} \dd \ell_\i(\cc_\k)\tt v_\k\leq 0$ (otherwise we can simply flip the sign of $v_\k$). 

Using Lagrange strong duality, the solution of \eqref{equ: update mgd} can be framed into 
\begin{align} \label{equ:mgd_dual}
v_\k & = \sum_{\i=1}^m \omega_{\i,\k} \dd \ell_\i(\cc_\k),
\\
\nonumber
\text{where }
  \{\omega_{\i,\k}\}_{\i=1}^m
  & =\arg\min_{\omega\in\C^{m}}\norm{ 
  \dd_\cc \ell_{\omega} (\cc_t) }.
\end{align}
It is easy to see from \eqref{equ:mgd_dual} 
that the set of fixed points of MDG (which satisfy $v_\k=0$) 
coincides with the Pareto stationary set $\P^*$. 

A key disadvantage of MGD, however, 
is that the Pareto point 
that it converges to 
depends on the initialization and other algorithm configurations in a rather implicated and complicated way. 
It is difficult to explicitly control MGD to make it converge to points with specific properties.

\section{Optimization In Pareto Set} \label{sec: problem}

The Pareto set typically contains an infinite number of points. 
In the \emph{optimization in Pareto set} (OPT-in-Pareto) problem,
we are given an extra criterion function $F(\theta)$ in addition to the objectives $\L$, and 
we want to minimize $F$ in the Pareto set of $\L$, that is, 
\begin{eqnarray} \label{equ: main_problem}
\min_{\th\in \P^*}F(\th).
\end{eqnarray}
For example, 
one can find the Pareto point whose loss vector $\L(\cc)$ is the closest to a given reference point $r\in \RR^m$ by choosing 
$F(\th) = \norm{\L(\th) - r }^2$. We can also design $F$ to encourages $\L(\cc)$  to be proportional to $r$, i.e., $\L(\cc)\propto r$; a constraint variant of this problem was considered in 
\citet{mahapatra2020multi}. %

We can further generalize 
OPT-in-Pareto %
to allow the criterion $F$ to depend on an ensemble of Pareto points $\{\th_1, ...,\th_N\}$ jointly, that is, %
\begin{eqnarray} \label{equ: main_problem_multi}
\min_{\th_{1},...,\th_{N}\in \P^*}F(\th_{1},...,\th_{N}).
\end{eqnarray}
For example, if $F(\th_1,\ldots, \th_N)$ measures the diversity among $\{\th_i\}_{i=1}^N$, then optimizing it provides a set of diversified  points inside the Pareto set $\P^*$ yielding a good approximation of $\P^*$.
An example of diversity measure is 
\begin{align} \label{eqn: energy}
F(\theta_1, \ldots, \theta_N ) & =E(\L(\cc_1), \ldots, \L(\cc_N)
),
\\
\nonumber
\text{with } E(\L_1, \ldots, \L_N)
& =\sum_{i\neq j}
\left\Vert \L_i-\L_j\right\Vert ^{-2},
\end{align}
where $E$ is known as an \emph{energy distance} in 
computational geometry, 
whose minimizer can be shown to give 
an uniform distribution on manifold asymptotically when $N\to\infty$ 
\citep{hardin2004discretizing}. 
This formulation is particularly useful when the users' preference is unknown during the training time, and we want to return an ensemble of models that well cover the different areas of the Pareto set to allow the users to pick up a model that fits their needs regardless of their preference. 
The problem of profiling Pareto set has attracted 
 a line of recent works  
\citep[e.g.,][]{lin2019pareto,mahapatra2020multi,ma2020efficient,deist2021multi}, but they
rely on specific criterion or heuristics and do not address the general optimization of form \eqref{equ: main_problem_multi}.  

\paragraph{Manifold Gradient Descent} 
One straightforward approach to 
OPT-in-Pareto is to deploy manifold gradient descent \citep{hillermeier2001generalized,bonnabel2013stochastic}, 
which conducts steepest descent of $F(\cc)$ 
in the Riemannian manifold formed by the Pareto set $\P^*$. Initialized at $\th_0\in \P^*$, manifold gradient descent updates $\th_{\k}$ at the $\k$-th iteration along the direction of the projection of $\nabla F(\th_\k)$ on the tangent space $\mathcal T(\th_\k)$ at $\th_\k$ in $\P^*$, %
\[
\th_{\k+1}=\th_\k-\xi\text{Proj}_{\mathcal T(\th_\k)}(\nabla F(\th_\k)).
\]
By using the stationarity characterization in \eqref{equ: pareto stationary}, under proper regularity conditions, 
one can show that the tangent space $\mathcal T(\th_\k)$ equals the null space of the Hessian matrix $\dd^2_{\cc} \ell_{\omega_\k}(\cc_\k)$, where $\omega_\k = \argmin_{\omega\in\C^m}\norm{\dd_{\cc}\ell_{\omega}(\cc_\k)}$. However, the key issue of manifold gradient descent is the high cost for calculating this null space of Hessian matrix. 
Although numerical techniques such as Krylov subspace iteration \citep{ma2020efficient} or conjugate gradient descent \citep{koh2017understanding} can be applied, 
the high computational cost (and the complicated implementation) still impedes its application in large scale deep learning problems.
See Section~\ref{sec:intro} for discussions on other related works. %

\section{Pareto Navigation Gradient Descent for OPT-in-Pareto} \label{sec: algo}
We now introduce our main algorithm, Pareto Navigating Gradient Descent (\PNG), which provides a practical approach to OPT-in-Pareto. 
For convenience, we focus on the single point problem in \eqref{equ: main_problem} in the presentation. 
The generalization to the multi-point problem in \eqref{equ: main_problem_multi} is straightforward.  
We first introduce the main idea and then present theoretical analysis in Section~\ref{sec: theory}. 

We consider the general incremental updating rule of form 
$$
\cc_{\k+1} \gets \cc_\k - \xi v_\k,
$$
where $\xi$ is the step size 
and $v_\k$ is an update direction that we shall choose to 
achieve the following desiderata in balancing the decent of $\{\ell_\i\}$ and $F$: 

i) When $\cc_\k$ is far away from the Pareto set, we want to choose $v_\k$ to give Pareto improvement to $\cc_\k$, moving it towards the Pareto set. The amount of Pareto improvement might depend on how far $\cc_\k$ is to the Pareto set.

ii) If the directions that yield Pareto improvement are not unique, we want to choose the Pareto improvement direction that decreases $F(\cc)$ most. 

iii) When $\cc_\k$ is very close to the Pareto set,  
e.g., having a small  $g(\cc)$, 
we want to fully optimize $F(\cc)$.

We achieve the desiderata above by using the $v_\k$ that solves the following optimization: 
\begin{align} \label{opt: relax}
& v_\k =\argmin_{v \in \mathbb{R}^n}\left \{  \frac12\left\Vert \nabla F(\th_t)-v \right\Vert ^{2} \right\}
\\
\nonumber
&\text{s.t.  }
\nabla_{\th}\ell_\i(\th_t) \tt v %
\ge \phi_\k,  ~~~~\forall \i\in [m], 
\end{align}
where we want $v_\k$ to be as close to $\nabla F(\th_t)$ as possible (hence decrease $F$ most), conditional on that the decreasing rate $\nabla_{\th}\ell_\i(\th_t) \tt v_t $  of all losses $\ell_\i$ 
are lower bounded by a \emph{control parameter} $\phi_\k$. 
A positive $\phi_\k$ enforces that $\nabla_{\th_t}\ell_\i(\th) \tt v_\k$ is positive for all $\ell_\i$, 
hence ensuring a Pareto  improvement when the step size is sufficiently small. The magnitude of $\phi_\k$ controls 
how much Pareto improvement we want to enforce, so we may want to gradually decrease $\phi_\k$ when we move closer to the Pareto set. 
In fact, varying $\phi_\k$ provides an intermediate updating direction between the vanilla gradient descent on $F$ and MGD on $\{\ell_\i\}$: 

i) If $\phi_\k = -\infty$, we have $v_\k = \dd F(\cc_\k)$ and it conducts a pure gradient descent on $F$ without considering $\{\ell_\i\}$.

ii) If $\phi_\k \to +\infty$, 
 then 
 $v_\k$ approaches to the MGD direction of 
 $\{\ell_i\}$ in \eqref{equ: update mgd} 
 without considering $F$. %

In this work, we propose to choose $\phi_\k$ based on the minimum gradient norm $g(\cc_\k)$ in \eqref{equ: pareto stationary} as a surrogate indication of Pareto local optimality. In particular, we consider the following simple design: 
\begin{align}\label{equ:phi}
\phi_\k = \begin{cases} 
-\infty & \text{if $g(\cc_\k) \leq \ep$},  \\
\alpha_\k g(\th_\k) & \text{if $g(\cc_\k) >  \ep$},
\end{cases} 
\end{align}
where $\ep$ is a small tolerance parameter 
and $\alpha_\k$ is a positive hyper-parameter.  
When $g(\cc_\k) > \ep$, 
we set $\phi_\k$ to be proportional to $g(\cc_\k)$, to ensure Pareto improvement based on how far $\cc_\k$ is to Pareto set. 
When $g(\cc_\k) \leq \ep$, 
we set $\phi_\k =-\infty$ which ``turns off'' the control and hence fully optimizes $F(\cc)$. %

\begin{algorithm*}[t]
\begin{algorithmic}[1]
\State{Initialize $\theta_0$; decide the
    step size $\xi$, and the control function $\phi$ in \eqref{equ:phi} (including the threshold $\ep >0$ and the descending rate $\{\alpha_\k\}$).}
\For{iteration $\k$}
    \vspace{-0.4cm}
    \bbb\label{equ:vk00}   \theta_{\k+1} \gets \theta_\k - \xi v_t, && 
 v_t = \nabla F(\theta_\k) + \textstyle{\sum}_{i=1}^m \lambda_{\i,\k} \nabla \ell_\i(\theta_\k) ,
    \vspace{-0.2cm}
    \eee 
    where $\lambda_{\i,\k} =0,~\forall \i\in[m]$ if $g(\th_\k) \leq \ep$, and $\{\lambda_{\i,\k}\}_{t=1}^m$ is the solution of (\ref{equ: dual}) with $\phi(\th_\k)=\alpha_\k g(\th_\k)$ when  $g(\th_\k) > \ep$.
\EndFor
\end{algorithmic}\caption{Pareto Navigating Gradient Descent}\label{alg:main}
\end{algorithm*}

In practice, the optimization in \eqref{opt: relax} can be  solved efficiently by  its dual form as follows. %
\begin{theorem} \label{thm: sol}
The solution $v_t$ of \eqref{opt: relax},  
if it exists, 
has a form of 
\bbb \label{equ:vk0}
v_\k = \dd F(\cc_\k) + \sum_{t=1}^m \lambda_{\i,\k} \dd \ell_\i(\cc_\k),
\eee 
with $\{\lambda_{\i,\k}\}_{t=1}^m$ %
the solution of the following dual problem 
\bbb \label{equ: dual}
\max_{\lambda\in\RRplus^m}-\frac{1}{2}|| \nabla F(\cc_\k)+\sum_{\i=1}^{m}\lambda_{\k}\nabla\ell_\i(\th_\k)|| ^{2}+\sum_{\i=1}^{m}\lambda_\i\phi_\k. 
\eee 
\end{theorem}

The optimization in \eqref{equ: dual} can be solved efficiently for a small $m$ (e..g, $m\leq 10$), which is the case for typical applications. 
We include the details of the practical implementation in Algorithm~\ref{alg:main}.

\section{Theoretical Properties} \label{sec: theory}

We provide a theoretical quantification on how {\PNG} 
guarantees to i) move the solution towards the Pareto set (Theorem~\ref{thm:odeell}); 
and ii) optimize $F$ in a neighborhood of Pareto set (Theorem~\ref{thm:odef}). 
To simplify the result and highlight the intuition, we focus on the continuous time limit of {\PNG},  
which yields a differentiation equation $\df  \cc_\tim = - v_\tim \df \tim$ with $v_\tim$ defined in \eqref{opt: relax}, 
where $t\in\RRplus$ is a continuous integration time. 

\begin{assumption}\label{asm:basic}
Let $\{\cc_t\colon t\in\RRplus\}$ be a solution of $\df \cc_t = -v_t \df t$ with $v_t$ in \eqref{opt: relax}; $\phi_k$ in \eqref{equ:phi};  $\ep>0$; and $\alpha_t \geq 0$,$\forall t\in \RRplus$. %
Assume $F$ and $\L$ are continuously differentiable on 
$\RR^\dimcc$, and lower bounded with 
$F\true\defeq \inf_{\cc\in \RR^\dimcc}F(\cc) > -\infty$ and $
\ell_\i\true \defeq \inf_{\cc\in \RR^\dimcc}\ell_\i(\cc)  > -\infty$. 
Assume $\sup_{\cc\in\RR^\dimcc}\norm{\dd F(\cc)}\leq c$. 
\end{assumption}

Technically, 
$\df \cc_t = -v_t \df t$ is a piecewise smooth dynamical  system whose solution should be taken in the Filippov sense using the notion of {differential inclusion} \citep{bernardo2008piecewise}.   
The solution always exists %
under mild regularity conditions 
although it may not be unique. 
Our results below apply to all  solutions.

\subsection{Pareto Optimization} 
We 
now show that the algorithm converges to the vicinity of Pareto set 
quantified by a notion of Pareto closure. 
For $\epsilon\geq 0$, 
let $\P_\epsilon$   
be the set of Pareto $\epsilon$-stationary points: 
$\P_{\epsilon} = \{\cc\in \RR^\dimcc \colon ~ g(\cc) \leq \epsilon\}$. %
The Pareto closure of a set $\P_{\epsilon}$, denoted by $\overline\P_\epsilon$ is the set of points 
that perform no worse than at least one point in $\P_{\epsilon}$, that is, %
\begin{align*}
\overline{\P}_{\epsilon}:= \cup_{\cc\in \P_{\epsilon}} \overline{\{\cc\}}, && 
\overline{\{\cc\}} = \{\cc'\in \RR^\dimcc\colon ~~ \L(\cc') \preceq \L(\cc)\}.
\end{align*}
Therefore, 
$\overline\P_\epsilon$ is better than  or at least as good as $\P_\epsilon$ 
in terms of Pareto efficiency.%

\begin{theorem}[Pareto Improvement on $\L$]  \label{thm:odeell}
 Under Assumption~\ref{asm:basic}, 
assume $\cc_0\not\in\P_\ep$, and $t_\ep$ is the first time when $\cc_{t_\ep} \in \P_\ep$, then for any time $t < t_\ep,$ 
\bb 
\frac{\df}{\df t} \ell_\i( \cc_t) \leq -\alpha_t g(\cc_t),\ 
\min_{s\in[0,t]} g(\cc_{\color{black}s}) \leq \frac{\min_{\i\in[m]}(\ell_\i(\cc_0)-\ell_i\true)}{\int_0^t \alpha_s \df s}.
\ee 
Therefore,  the update yields Pareto improvement on $\L$ when $\cc_t \not\in \P_\ep$ and  $\alpha_t g(\cc_t)>0$.

Further, 
if $\int_0^t \alpha_s \df s = +\infty$, 
then for any $\epsilon>\ep$,  there exists a finite time $t_\epsilon \in \RRplus$ on which the solution enters $\P_{\epsilon}$ and stays within $\overline \P_\epsilon$ afterwards, that is, we have $\cc_{t_\epsilon} \in \P_\epsilon$ and $\cc_t\in \overline \P_\epsilon$ 
for any $t \geq t_\epsilon$. 
\end{theorem}
Here we guarantee that $\cc_t$ must enter
 $\P_\epsilon$ for some time (in fact infinitely often), 
but it is not confined in $\P_\epsilon$. %
On the other hand, 
$\cc_t$ does not leave $\overline \P_\epsilon$ after it first enters $\P_\epsilon$ thanks to the Pareto improvement property. 

\subsection{Criterion Optimization}
We now show that {\PNG} finds a local optimum of $F$ inside the Pareto closure $\overline \P_{\epsilon}$ in an approximate sense. 
We first show 
that a fixed point $\cc$ of the algorithm that is locally convex on $F$ and $\L$ must be a local optimum of $F$ in the Pareto closure of $\{\cc\}$, and then quantify the convergence of the algorithm. %
\begin{theorem}[PNG Finds Local Optimum]\label{lem:dfjijgfgfgfgifjg} 
Under Assumption~\ref{asm:basic}, we have

If $\cc_t \not\in \P_\ep$  
is a fixed point of the algorithm, that is, $\frac{\df \theta_t}{\df t} = -v_t = 0$, and $F$, $\L$ are  convex in a neighborhood $\cc_t$, then 
$\cc_t$ is a local minimum of $F$
in the Pareto closure $\overline{\{\theta_t\}}$, %
that is,  there exists a neighborhood of $\cc_t$ in which  there exists no point $\cc'$ such that $F(\cc') < F(\cc_t)$ and $\L(\cc') \preceq \L(\cc_t)$. 

If $\cc_t\in \P_\ep$, we have 
$v_t = \dd F(\cc_t)$, 
and hence a fixed point with $\frac{\df \theta_t}{\df t} = -v_t = 0$ is an unconstrained local minimum of $F$ when $F$ is locally convex on $\cc_t$. 
\end{theorem}
\begin{theorem}[Convergence] \label{thm:odef}
Let  $\epsilon > \ep$ and assume $g_{\epsilon} \defeq \sup_{\cc} \{g(\cc) \colon ~\cc\in \overline \P_\epsilon\}<+\infty$ and  $\sup_{t\geq0}\alpha_t<\infty$.   
Under Assumption~\ref{asm:basic}, 
when we initialize from $\cc_0 \in \P_\epsilon$, %
we have 
$$
\min_{s\in[0,t]}\norm{ \frac{\df\cc_s}{\df s} }^2 \leq 
\frac{F(\cc_0)-F\true}{t} + \frac{1}{t}\int_{0}^t 
\alpha_s \left (\alpha_s {g_\epsilon}    + 
c \sqrt{g_\epsilon}  \right ) 
\df s. 
$$
In particular, 
if we have $\alpha_t = \alpha = const$, then 
$\min_{s\in[0,t]}\norm{\df \theta_s/\df s}^2  = \bigO\left (1/t + \alpha \sqrt{g_\epsilon} \right ).$

If  
$%
\int_0^\infty\alpha_t^\gamma\df t <+\infty$ for some $\gamma \geq 1$, we have 
$\min_{s\in[0,t]}\norm{\df \theta_s/\df s}^2  = \bigO(1/t + \sqrt{g_\epsilon}/t^{1/\gamma}).$
\end{theorem}

Combining the results in Theorem~\ref{thm:odeell} and \ref{thm:odef}, we can see that the choice of   sequence $\{\alpha_t\colon t\in \RRplus \}$ controls how fast we want to decrease $\L$ vs. $F$. Large $\alpha_t$ yields faster descent on $\L$, but slower descent on $F$. 
Theoretically, 
using a sequence that satisfies $\int\alpha_t\df t =+\infty$ and $\int \alpha_t^\gamma \df t<+\infty$ for some $\gamma>1$   allows us to ensure that both $\min_{s\in[0,t]} g(\cc_s)$ and $\min_{s\in[0,t]} \norm{\df \theta/\df s}^2$ converge to zero. 
If we use a  constant sequence $\alpha_\tim =\alpha$, 
it introduces an $\bigO(\alpha \sqrt{g_\epsilon})$ term that does not vanish as $t\to +\infty$. However, 
we can expect that $g_\epsilon$ is  small when $\epsilon$ is small
for well-behaved functions. %
In practice, we find that constant $\alpha_\tim$
 works sufficiently well.

\section{Related Work} \label{sec: review}
\paragraph{Optimization Algorithms for MOO}
There has been a rising interest in MOO in deep learning, 
mostly in the context of multi-task learning.
But most existing methods can not be applied to 
the general OPT-in-Pareto problem. %
A large body of recent works focus on improving non-convex optimization  for finding \emph{some} model in the Pareto set, 
but cannot search for a \emph{special} model satisfying a specific criterion %
\citep{chen2018gradnorm,kendall2018multi,NEURIPS2018_432aca3a,yu2020gradient,NEURIPS2020_16002f7a,Wu2020Understanding,fifty2020measuring,javaloy2021rotograd}. %

\paragraph{Specific Instantiations of OPT-in-Pareto}
One previous work \citep{mahapatra2020multi} and two concurrent works \citep{kamani2021pareto,chen2021weighted} study specific instantiations of the general OPT-in-Pareto problem and thus are highly related to this paper. We give a detailed review. \citet{mahapatra2020multi} aims to search Pareto model that satisfies a constraint on the ratio between the different objectives, which can be viewed as OPT-in-Pareto problem when the criterion $F$ is a proper measure of constraint violation (i.e, the non-uniformity score defined in \citet{mahapatra2020multi}). EPO, the proposed algorithm in \citet{mahapatra2020multi} heavily relies on a special property of the ratio constraint problem: there always exists an updating direction that either gives Pareto improvement or reduces the constraint violation or both. However, a general OPT-in-Pareto problem does not have such nice property, making EPO only a specialized algorithm for the ratio constraint problem rather than a general OPT-in-Pareto problem. In section \ref{sec: subset application} we demonstrate that PNG is able to recover the functionality of EPO while being a more general algorithm for OPT-in-Pareto. \citep{kamani2021pareto} formulate the fairness learning as a MOO problem in which the accuracy and fairness measure are considered as the two objectives. It first proposes PDO, an algorithm that converges to Pareto stationary set by viewing MOO as a bi-level optimization (which is a standard MOO algorithm that does not solve any instance of OPT-in-Pareto) and then BP-PDO, an modification of PDO that seeks a Pareto model that satisfies the ratio-constraint considered in \citet{mahapatra2020multi}. Admittedly, it is possible to extend the BP-PDO for general OPT-in-Pareto problems but such extension is non-trivial: even for the special ratio-constraint problem, it is unclear what convergence and optimality guarantee BP-PDO has (only guarantee of PDO is given in \citet{kamani2021pareto}). In comparison, our PNG is shown to converge to the local optimum of OPT-in-Pareto problem. \citet{chen2021weighted} aims to pre-train a multi-task model such that the representations of the tasks are similar. Their problem is essentially an OPT-in-Pareto problem where the discrepancy of task representations are chosen as the criterion function. Compared with PNG, the proposed TAWT algorithm requires the computation of inverse Hessian product at each iteration making its computational cost large.

\paragraph{Approximation of Pareto Set}
There has been increasing interest in
finding a compact approximation of the Pareto set. \citet{navon2020learning, lin2020controllable} use hypernetworks to approximate the map from linear scalarization weights to the corresponding Pareto solutions; these methods could not fully profile non-convex Pareto fronts due to the limitation of linear scalarization \citep{boyd2004convex}, and the use of hypernetwork introduces extra optimization difficulty. 
Another line of works \citep{lin2019pareto,mahapatra2020multi} approximate  the Pareto set by Pareto models with different user preference vectors that rank the relative importance of different tasks; these methods need a good heuristic design of preference vectors, which requires prior knowledge of the Pareto front. 
\citet{ma2020efficient} leverages  manifold gradient to conduct a local random walk on the Pareto set but suffers from the high computational cost. \citet{deist2021multi} approximates the Pareto set by maximizing hypervolume, which also requires prior knowledge for a careful choice of good reference vector. \citet{liu2021profiling} introduces a repulsive force to encourage the model diversity without hurting their Pareto Optimality.

\paragraph{Applications of MOO}
Multi-task learning can also be applied to improve the learning in many other domains including domain generalization \citep{dou2019domain,Carlucci_2019_CVPR,albuquerque2020improving}, domain adaption \citep{sun2019unsupervised,luo2021learnable}, model uncertainty \citep{NEURIPS2019_a2b15837,zhang2020auxiliary,xie2021innout}, adversarial robustness \citep{yang2020multitask} and semi-supervised learning \citep{NEURIPS2020_06964dce}. All of those applications utilize a linear scalarization to combine the multiple objectives and it is thus interesting to apply the proposed OPT-in-Pareto framework, which we leave for future work.

\begin{figure*}[t]
\begin{centering}
\vspace{-0.1cm}
\hspace{-10mm}
\includegraphics[width=.28\textwidth]{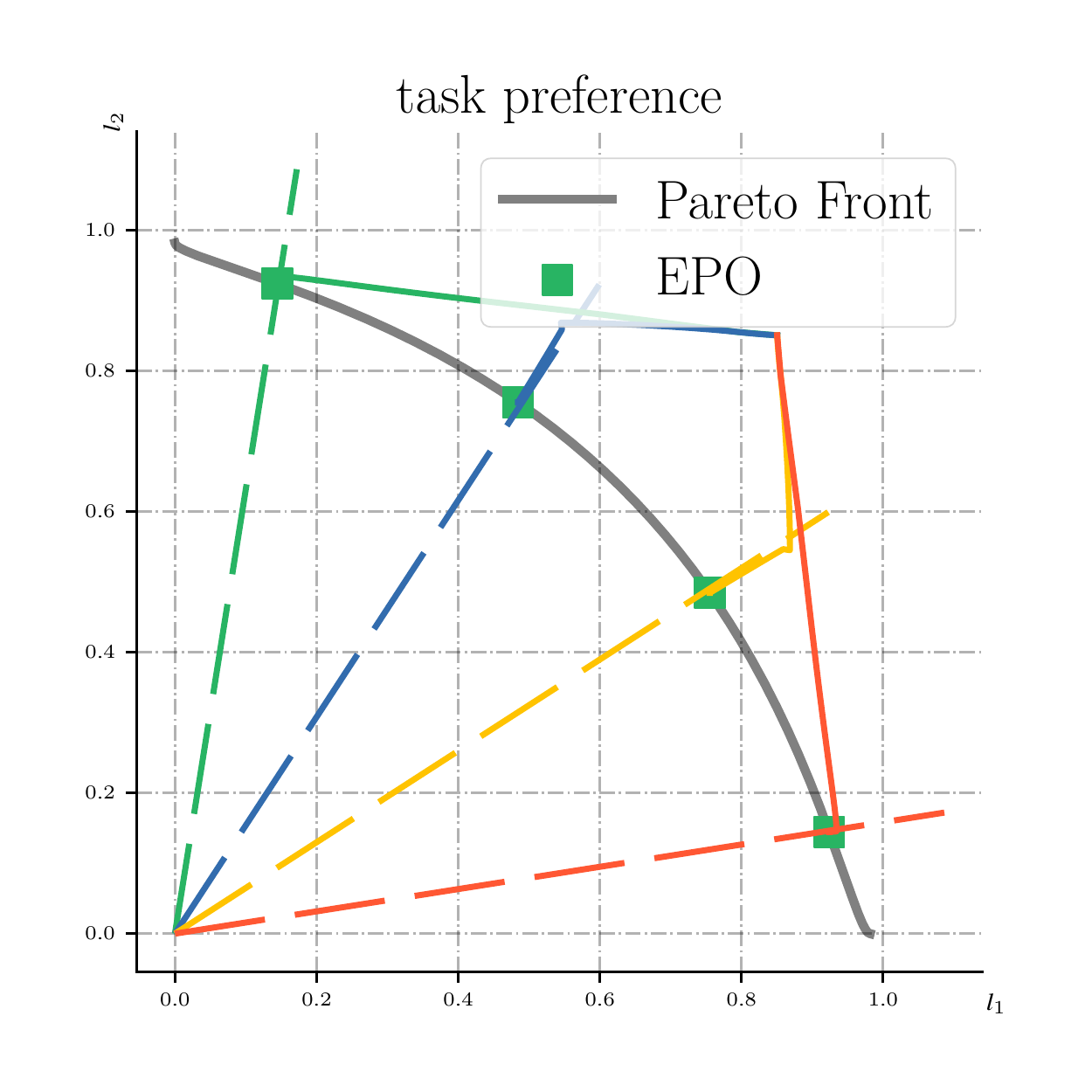}\hspace{-5mm}
\includegraphics[width=.28\textwidth]{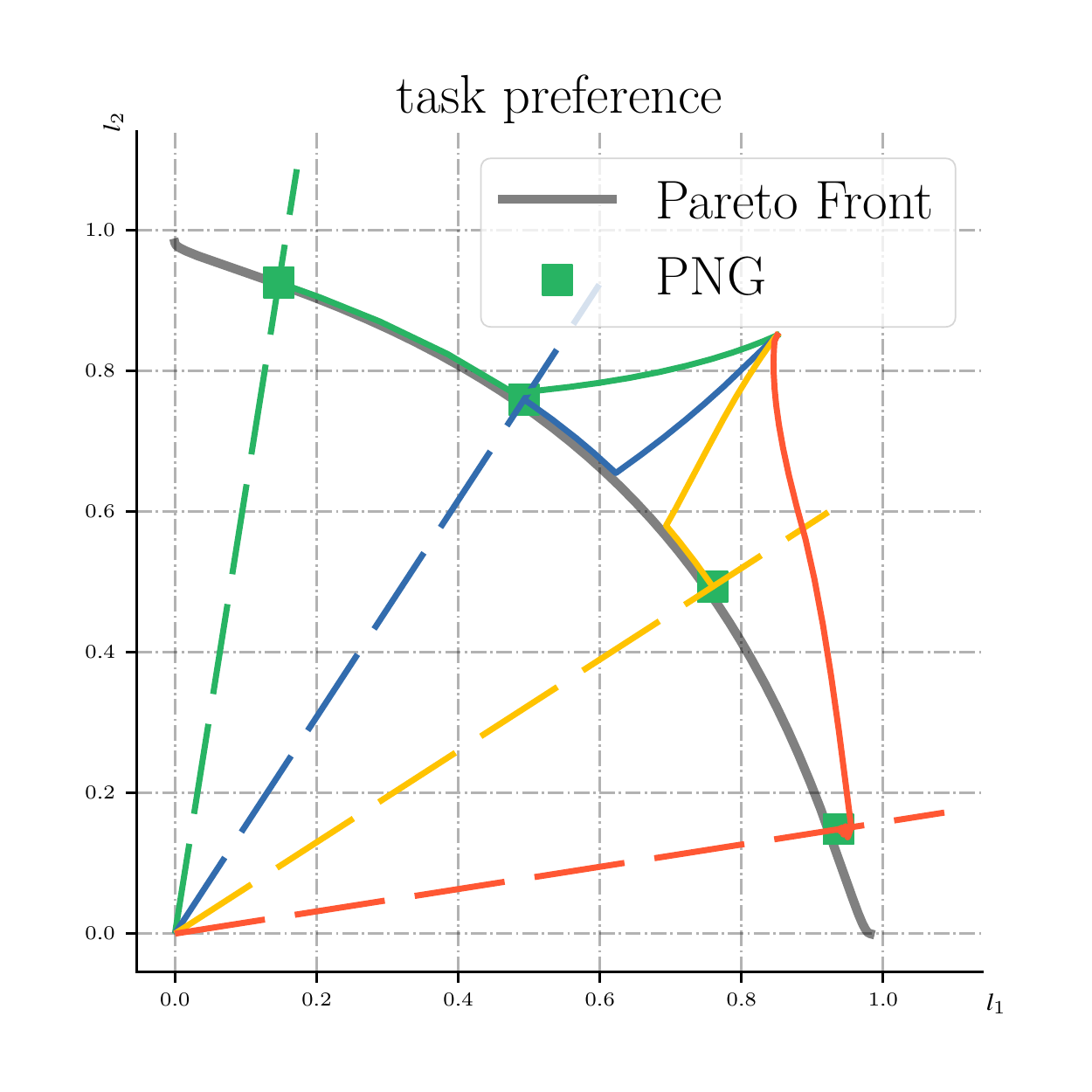}\hspace{-5mm}
\includegraphics[width=.28\textwidth]{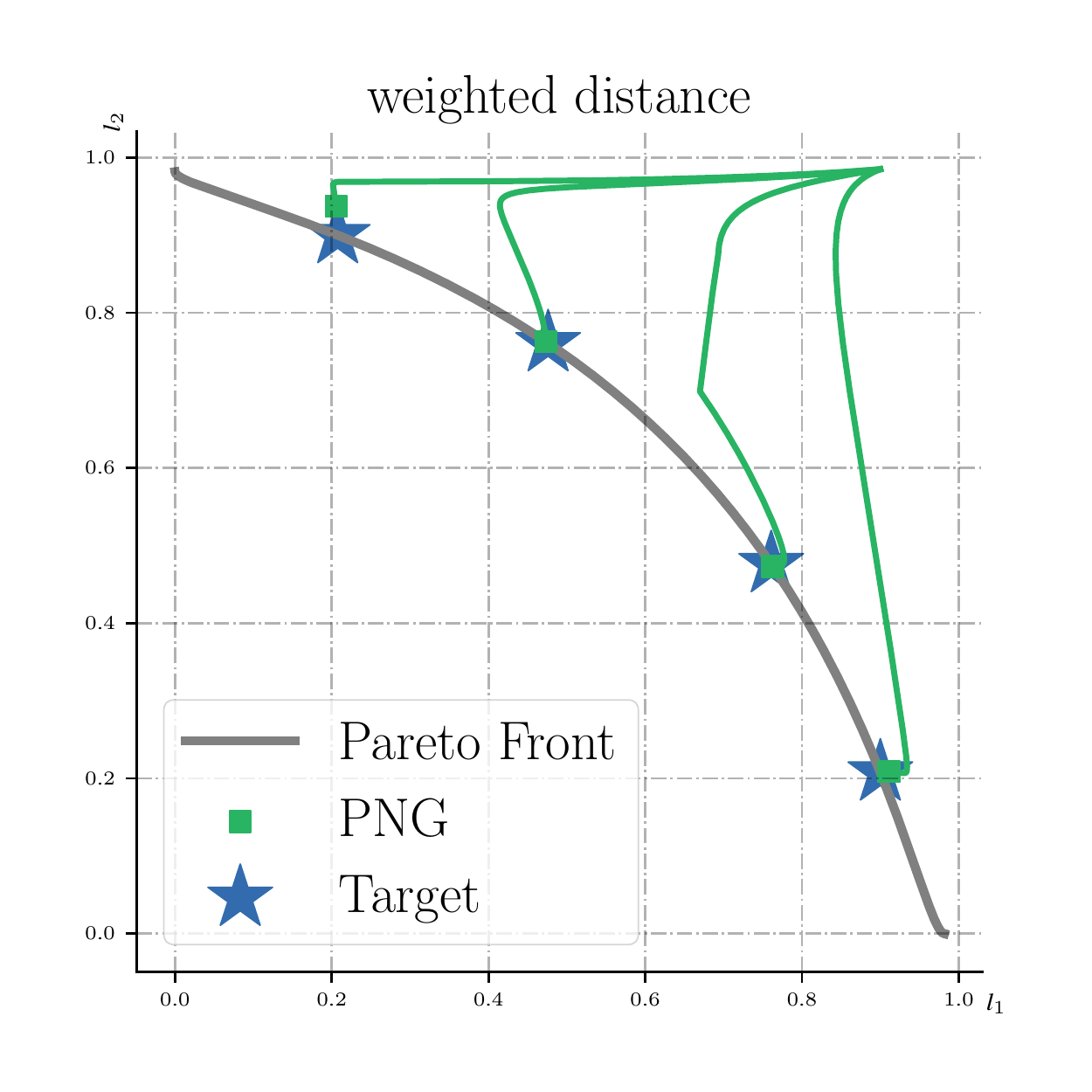}\hspace{-5mm}
\includegraphics[width=.28\textwidth]{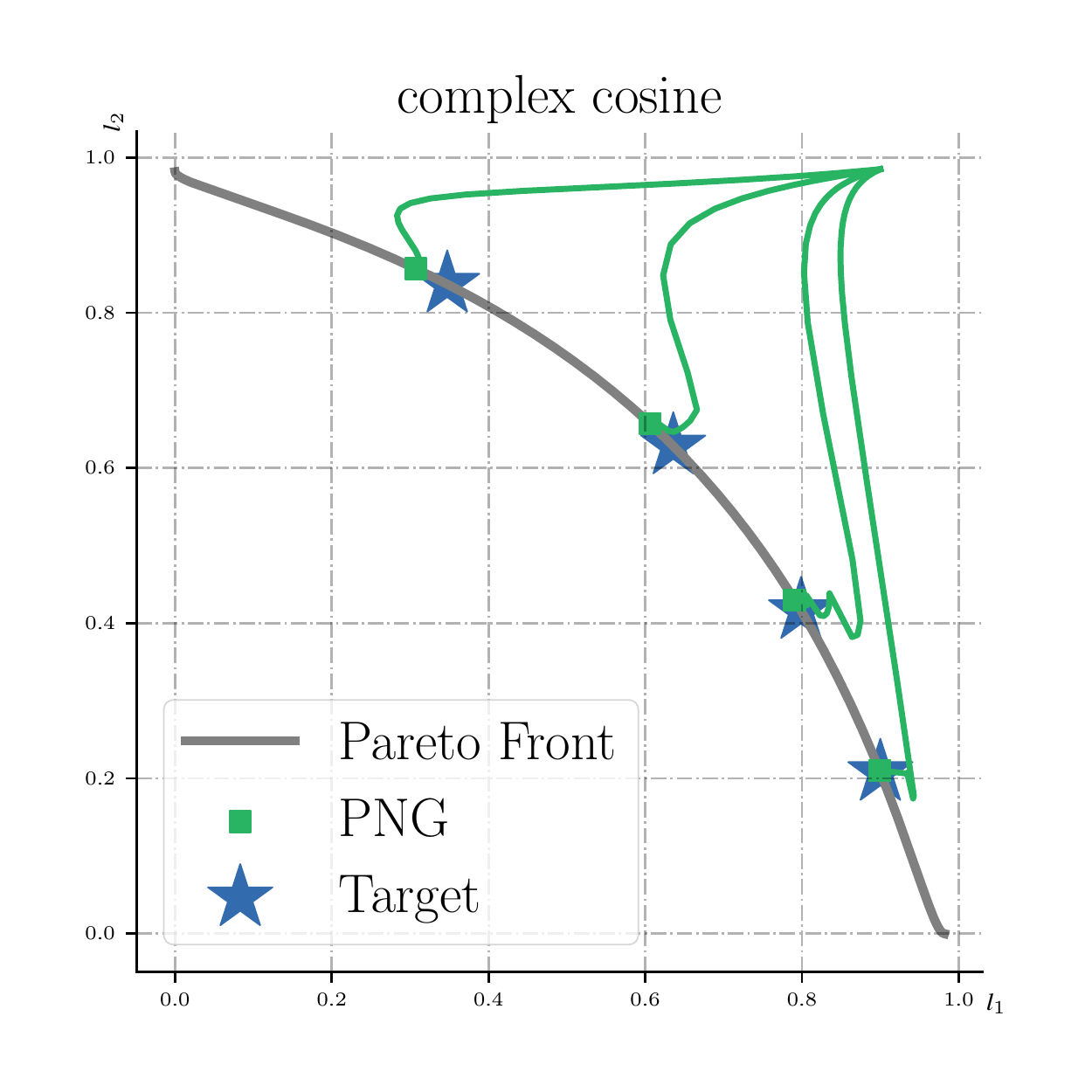}\hspace{-10mm}
\par\end{centering}
\vspace{-0.4cm}
\begin{tabular}{cccc}
\hspace{.1\textwidth}
\small (a) \hspace{.19\textwidth} & 
\small (b) \hspace{.175\textwidth} & 
\small (c)  \hspace{.175\textwidth} & 
\small (d)  \hspace{.19\textwidth}
\end{tabular}
\vspace{-1.2\baselineskip}
\caption{(a)-(b): the trajectory of finding Pareto models that satisfy different ratio constraints (shown in different colors) on the two objectives $\ell_1,\ell_2$ using EPO and PNG; 
we can see that PNG can achieve the same goal as EPO (with different trajectories) while being a more general approach. 
(c)-(d): the trajectory of finding Pareto models that minimize the weighted distance and complex cosine criterion using PNG. The green dots indicate the converged models. We can see that PNG can successfully locate the correct Pareto models that minimize different criteria.}\label{fig: epo recover}
\vspace{-0.5cm}
\end{figure*}

\section{Empirical Results} 
We introduce three applications of OPT-in-Pareto with {\PNG}: Singleton Preference, Pareto approximation and improving multi-task based domain generalization method. We also conduct additional study on how the learning dynamics of PNG changes with different hyper-parameters ($\alpha_t$ and $\ep$), which are included in Appendix \ref{appendix_sec: dynamics}. Other additional results that are related to the experiments in Section \ref{sec: subset application} and \ref{sec: exp approx} and are included in the Appendix will be introduced later in their corresponding sections. Code is available at \url{https://github.com/lushleaf/ParetoNaviGrad}.
\subsection{Finding Preferred Pareto Models} \label{sec: subset application}
We consider the synthetic example used in \citet{lin2019pareto, mahapatra2020multi}, which consists of two losses:
$\ell_{1}(\th)=1-\exp(-\left\Vert \th-\eta \right\Vert ^{2})$ and $\ell_{2}(\th)=1-\exp(-\left\Vert \th+\eta \right\Vert ^{2})$, where
$\eta = n^{-1/2}$ and 
$\dimcc=10$ is dimension of the parameter $\cc$. 

\paragraph{Ratio-based Criterion}
We first show that {\PNG} can solve the search problem under the ratio constraint of objectives in \citet{mahapatra2020multi}, i.e., finding a point $\th\in \P^*\cap \Omega$ with $\Omega=\{\th:r_{1}\ell_{1}(\th)=r_{2}\ell_{2}(\th)=...=r_{m}\ell_{m}(\th)\}$, given some preference vector $r=[r_1,...,r_m]$. We apply {\PNG} with the non-uniformity score defined in \citet{mahapatra2020multi} as the criterion, and compare with their algorithm called exact Pareto optimization (EPO).  
We show in  Figure \ref{fig: epo recover}(a)-(b) the trajectory of {\PNG} and EPO 
 for searching models with different preference vector $r$, starting from the same randomly initialized point. 
Both {\PNG} and EPO converge to the correct solutions but with different trajectories. {\color{black} This suggests that PNG is able to achieve the same functionality of finding ratio-constraint Pareto models as \citet{mahapatra2020multi,kamani2021pareto} do but being versatile to handle general criteria.} We refer readers to Appendix \ref{appendix_sec: ratio-based preference} for more results with different choices of hyper-parameters and the experiment details.

\paragraph{Other Criteria} 
We demonstrate that {\PNG} is able to find solutions for general choices of $F$. We consider the following designs of $F$:
1) weighted $\ell_2$ distance w.r.t. a reference vector $r\vv\in \RRplus^m$, that is, 
$F_{\text{wd}}(\th)=\sum_{i=1}^m (\ell_{i}(\th)-r_i)^{2}/r_i$; 
and 2) complex cosine: in which $F$ is a complicated function related to the cosine of task objectives, i.e., $F_{\text{cs}}=-\cos\left(\pi(\ell_{1}(\th)-r_1)/2\right) +(\text{cos}(\pi(\ell(\cc_{2})-r_2)) + 1)^2$. {\color{black} Here the weighted $\ell_2$ distance can be viewed as finding a Pareto model that has the losses close to some target value $r$, which can be viewed as an alternative approach to partition the Pareto set. The design of complex cosine aims to test whether PNG is able to handle a very non-linear criterion function.} In both cases, we take 
$r_1 =[0.2, 0.4, 0.6, 0.8]$ and $r_2 = 1-r_1$.  
We show in Fig \ref{fig: epo recover}(c)-(d) 
the trajectory of {\PNG}. As we can see, {\PNG} is able to correctly find the optimal solutions of OPT-in-Pareto. We also test {\PNG} on a more challenging ZDT2-variant used in \citet{ma2020efficient} and a larger scale MTL problem \citep{liu2019end}, for which we refer readers to Appendix \ref{appendix_sec: zdt} and \ref{appendix_sec: mtan}. 

\begin{table*}[t]
\begin{centering}
\vspace{-0.3cm}
\scalebox{1.0}{
\begin{tabular}{c|c|cc|cc}
\toprule
\multirow{2}{*}{Data} & \multirow{2}{*}{Method} & \multicolumn{2}{c|}{Loss} & \multicolumn{2}{c}{Acc}\tabularnewline
 &  & HV$\uparrow$ ($10^{-2}$) & IGD+$\downarrow$ ($10^{-2}$) & HV$\uparrow$ ($10^{-2}$) & IGD+$\downarrow$ ($10^{-2}$)\tabularnewline
\hline 
\multirow{4}{*}{Multi-MNIST} & Linear & $7.48\pm0.11$ & $0.142\pm0.034$ & $9.27\pm0.024$ & $0.036\pm0.0084$\tabularnewline
 & MGD & $7.69\pm0.10$ & $0.051\pm0.011$ & $9.27\pm0.023$ & $0.008\pm0.0010$\tabularnewline
 & EPO & $\pmb{7.87\ensuremath{\pm}0.16}$ & $0.069\pm0.028$ & $9.17\pm0.032$ & $0.065\pm0.0181$\tabularnewline
 & {\PNG} & $\pmb{7.86\ensuremath{\pm}0.11}$ & $\pmb{0.042\ensuremath{\pm}0.012}$ & $\pmb{9.39\ensuremath{\pm}0.036}$ & $\pmb{0.006\ensuremath{\pm}0.0022}$\tabularnewline
\hline 
\multirow{4}{*}{Multi-Fashion} & Linear & $0.38\pm0.059$ & $0.127\pm0.013$ & $4.76\pm0.019$ & $0.064\pm0.012$\tabularnewline
 & MGD & $0.42\pm0.064$ & $0.046\pm0.016$ & $4.77\pm0.019$ & $\pmb{0.023\ensuremath{\pm}0.003}$\tabularnewline
 & EPO & $0.36\pm0.058$ & $0.308\pm0.109$ & $4.78\pm0.030$ & $0.211\pm0.020$\tabularnewline
 & {\PNG} & $\pmb{0.47\ensuremath{\pm}0.066}$ & $\pmb{0.016\ensuremath{\pm}0.002}$ & $\pmb{4.81\ensuremath{\pm}0.021}$ & $\pmb{0.023\ensuremath{\pm}0.003}$\tabularnewline
\hline 
\multirow{4}{*}{Fashion-MNIST} & Linear & $5.01\pm0.057$ & $0.167\pm0.054$ & $8.46\pm0.046$ & $0.110\pm0.035$\tabularnewline
 & MGD & $5.09\pm0.069$ & $0.060\pm0.029$ & $8.40\pm0.045$ & $\pmb{0.049\ensuremath{\pm}0.011}$\tabularnewline
 & EPO & $4.60\pm0.166$ & $0.233\pm0.054$ & $8.12\pm0.041$ & $0.385\pm0.077$\tabularnewline
 & {\PNG} & $\pmb{5.27\ensuremath{\pm}0.054}$ & $\pmb{0.048\ensuremath{\pm}0.027}$ & $\pmb{8.53\ensuremath{\pm}0.047}$ & $\pmb{0.046\ensuremath{\pm}0.022}$\tabularnewline
\bottomrule
\end{tabular}
}
\par\end{centering}
\caption{
Results of approximating the Pareto set 
by different methods on three MNIST benchmark datasets. The numbers in the table are the averaged value and the standard deviation. 
Bolded values indicate the statistically significant best result with p-value less than 0.5 based on matched pair t-test.} \label{tbl: mnist}
\end{table*}

\begin{table*}
\begin{centering}
\scalebox{1.0}{
\begin{tabular}{c|cccc|c}
\toprule
PACS & art paint & cartoon & sketches & photo & Avg\tabularnewline
\hline 
D-SAM & $0.7733$ & $0.7243$ & $0.7783$ & $0.9530$ & $0.8072$\tabularnewline
DeepAll & $0.7785$ & $0.7486$ & $0.6774$ & $0.9573$ & $0.7905$\tabularnewline
\hline 
JiGen & $0.8009\pm0.004$ & $0.7363\pm0.007$ & $0.7046\pm0.013$ & $\pmb{0.9629\ensuremath{\pm}0.002}$ & $0.8012\pm0.002$\tabularnewline
JiGen+adv & $0.7923\pm0.006$ & $0.7402\pm0.004$ & $0.7188\pm0.005$ & $0.9617\pm0.001$ & $0.8033\pm0.001$\tabularnewline
JiGen+PNG & $\pmb{0.8014\ensuremath{\pm}0.005}$ & $\pmb{0.7538\ensuremath{\pm}0.001}$ & $\pmb{0.7222\ensuremath{\pm}0.006}$ & $\pmb{0.9627\ensuremath{\pm}0.002}$ & $\pmb{0.8100\ensuremath{\pm}0.005}$\tabularnewline
\bottomrule 
\end{tabular}
}
\par\end{centering}
\caption{Comparing different methods for domain generalization on PACS using ResNet-18. The values in table are the testing accuracy with its  standard deviation. The bolded values are the best models with p-value less than $0.1$ based on match-pair t-test.}\label{tbl: domain_small}
\vspace{-0.3cm}
\end{table*}

\subsection{Finding Diverse Pareto Models} \label{sec: exp approx}

\begin{figure}
\begin{centering}
\includegraphics[scale=0.13]{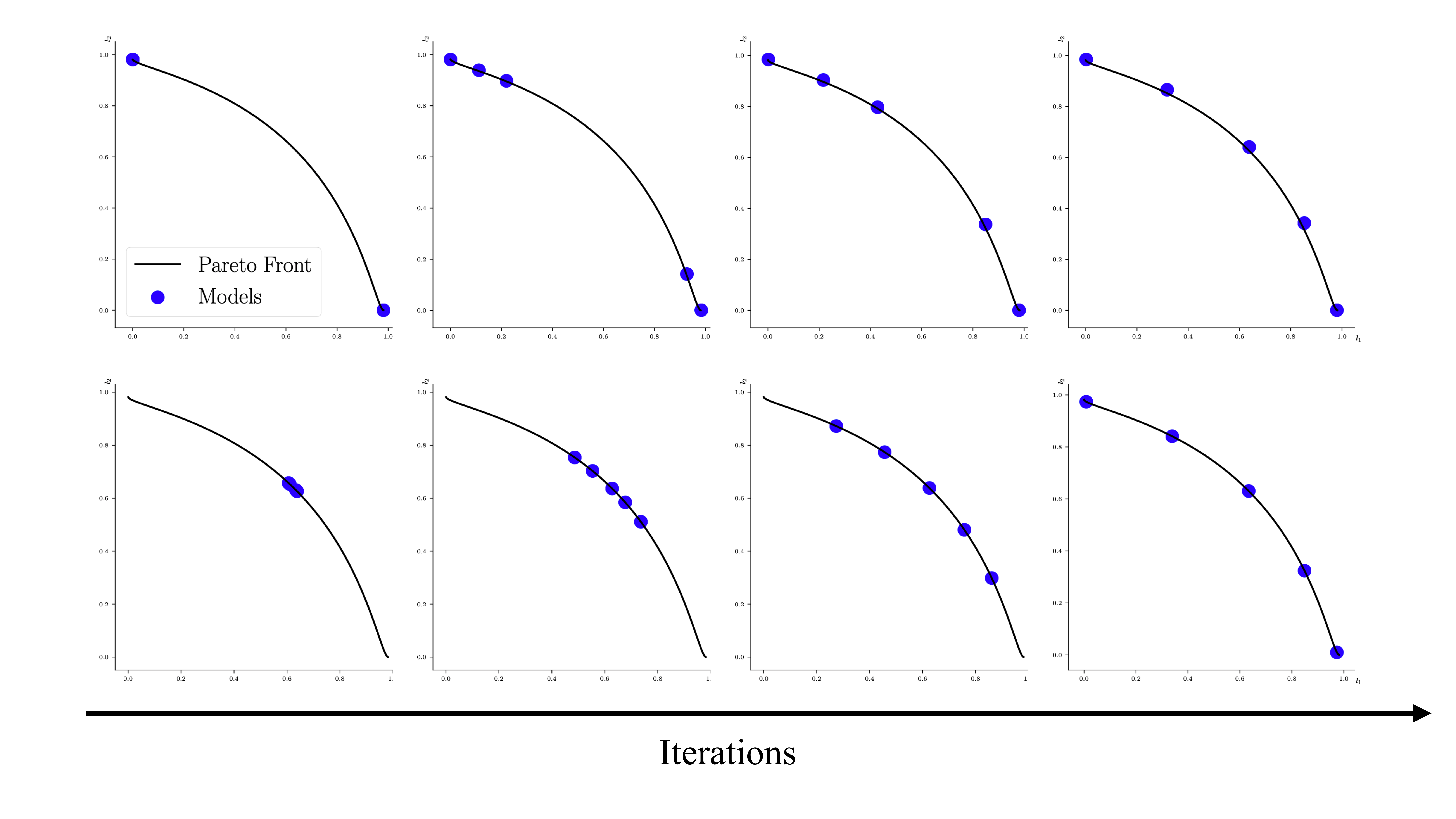}
\end{centering}
\vspace{-0.3cm}
\caption{Evolution of models from different initializations. Upper row starts with models at the boundary of the Pareto set. Lower row considers clustered initializations.} \label{fig: engergy_toy}
\vspace{-0.5cm}
\end{figure}

\paragraph{Synthetic Examples}
We reuse the synthetic example introduced in Section \ref{sec: subset application}. We consider learning 5 models to approximate the Pareto front staring from two types of extremely bad initializations.  Specifically, in the upper row of Figure \ref{fig: engergy_toy}, we consider initializing the models using linear scalarization. Due to the concavity of the Pareto front, linear scalarization can only learns models at the two extreme end of the Pareto front. The second row uses MGD for initialization and the models is scattered at an small region of the Pareto front. Different from the algorithm proposed by \citet{lin2019pareto} which relies on a good initialization, using the proposed energy distance function, PNG pushes the models to be equally distributed on the Pareto Front without the need of any prior information of the Pareto front even with extremely bad starting point.

\paragraph{Multi-MNIST Benchmark} We consider the problem of finding diversified points from the Pareto set by minimizing the energy distance criterion in \eqref{eqn: energy}. 
We use the same setting as \citet{lin2019pareto,mahapatra2020multi}. We consider three benchmark datasets: (1) MultiMNIST, (2) MultiFashion, and (3) MultiFashion+MNIST. For each dataset, there are two tasks (classifying the top-left and bottom-right images). We consider LeNet with multihead and train $N=5$ models to approximate the Pareto set. For baselines, we compare with 
linear scalarization, MGD \citep{NEURIPS2018_432aca3a}, and EPO \citep{mahapatra2020multi}. %
For the MGD baseline, we find that naively running it leads to poor performance as the learned models are not diversified and thus we initialize the MGD with 60-epoch runs of linear scalarization with equally distributed preference weights and runs MGD for the later 40 epoch. We refer the reader to Appendix \ref{appendix_sec: pareto approximation exp} for more details of the experiments.

We measure the quality of how well the found models $\{\theta_1,\ldots,\theta_N\}$ approximate the Pareto set using two standard metrics: Inverted Generational Distance Plus (IGD+) \citep{ishibuchi2015modified} and  hypervolume (HV) \citep{zitzler1999multiobjective}; see Appendix \ref{appendix_sec: pareto approximation metric} for their definitions. We run all the methods with 5 independent trials and report the averaged value and its standard deviation in Table \ref{tbl: mnist}. We report the scores calculated based on loss (cross-entropy) and accuracy on the test set. The bolded values indicate the best result with p-value less than 0.05 (using matched pair t-test). In most cases, {\PNG} improves the baselines by a large margin. We include ablation studies in Appendix \ref{appendix_sec: pareto approximation abl} and additional comparisons with the second-order approach proposed by \citet{ma2020efficient} in Appendix \ref{appendix_sec: compare second order}.

\subsection{Application to Multi-task based Domain Generalization Algorithm}
JiGen \citep{carlucci2019domain} learns a domain generalizable model by learning two tasks based on linear scalarization, which essentially searches for a model in the Pareto set and requires choosing the weight of linear scalarization carefully. It is thus natural to study whether there is a better mechanism that dynamically adjusts the weights of the two losses so that we eventually learn a better model. Motivated by the
adversarial feature learning \citep{JMLR:v17:15-239}, we propose to improve JiGen such that the latent feature representations of the two tasks are well aligned. This can be framed into an OPT-in-Pareto problem where the criterion is the discrepancy of the latent representations (implemented using an adversarial discrepancy module in the network) of the two tasks. PNG is applied to solve the optimization. We evaluate the methods on PACS \citep{Li_2017_ICCV}, which covers 7 object categories and 4 domains (Photo, Art Paintings, Cartoon, and Sketches). The model is trained on three domains and tested on the rest of them. Our approach is denoted as JiGen+PNG and we also include JiGen + adv, which simply adds the adversarial loss as regularization and two other baseline methods (D-SAM \citep{d2018domain} and DeepAll \citep{carlucci2019domain}). For the three JiGen based approaches, we run 3 independent trials and for the other two baselines, we report the results in their original papers. Table \ref{tbl: domain_small} shows the result using ResNet-18, which  demonstrates 
the improvement by the application of the OPT-in-Pareto framework. We also include the results using AlexNet in the Appendix. Please see Appendix \ref{appendix_sec: dg} for the additional results and more experiment details.

\section{Conclusion}
This paper studies the OPT-in-Pareto, a problem that has been studied in operation research with restrictive linear or convexity assumption but largely under-explored in deep learning literature, in which the objectives are non-linear and non-convex. Applying algorithms such as manifold gradient descent requires eigen-computation of the Hessian matrix at each iteration and thus can be expensive. We propose a first-order approximation algorithm called Pareto Navigation Gradient Descent (PNG) with theoretically guaranteed descent and convergence property to solve OPT-in-Pareto. %

\bibliography{ref}

\appendix
\onecolumn
\title{Appendix for Optimization in Pareto Set: Searching Special Pareto Models}
\maketitle

\section{Theoretical Analysis}

\paragraph{Theorem~\ref{thm: sol} {[Dual of \eqref{opt: relax}]}}
\emph{The solution $v_t$ of \eqref{opt: relax}, 
if it exists, has a form of 
\bb%
v_\tim = \dd F(\cc_\tim) + \sum_{\i=1}^m \lambda_{i,t} \dd \ell_\i(\cc_\tim),
\ee%
with $\{\lambda_{i,t}\}_{\i=1}^m$ %
the solution of the following dual problem 
\bb%
\max_{\lambda\in\RRplus^m}-\frac{1}{2}\left\Vert \nabla F(\cc_\tim)+\sum_{\i=1}^{m}\lambda_\tim\nabla\ell_\i(\th_\tim)\right\Vert ^{2}+\sum_{\i=1}^{m}\lambda_\i\phi_\tim,
\ee%
where $\RRplus^m$ is the set of nonnegative  $m$-dimensional vectors, that is, 
$\RRplus^m = \{\lambda\in\RR^m\colon \lambda_\i\geq 0,~~\forall \i\in[m]\}$. 
}
\begin{proof}
By introducing Lagrange multipliers, the optimization in \eqref{opt: relax} is equivalent to  the following minimax problem: 
\[
\min_{v\in \RR^\dimcc}\max_{\lambda\in \RR_+^m }\frac{1}{2}\left\Vert \nabla F(\th_\tim)-v\right\Vert  ^{2}+\sum_{\i=1}^{m}\lambda_\i\left( \phi_\tim -%
\nabla\ell_\i(\th_\tim)
\tt v
\right). %
\]
With strong duality of convex quadratic programming (assuming the primal problem is feasible), %
we can exchange the order of min and max, yielding  
\begin{align*}
\max_{\lambda\in \RR_+^m }
\left\{ \Phi(\lambda)\defeq \min_{v\in \RR^\dimcc}
\frac{1}{2}\left\Vert \nabla F(\th_\tim)-v\right\Vert  ^{2}+\sum_{\i=1}^{m}\lambda_\i\left( \phi_\tim -%
\nabla\ell_\i(\th_\tim)
\tt v
\right) \right\}. %
\end{align*}
It is easy to see that the 
minimization w.r.t. $v$ is achieved 
when $v =\nabla F(\cc_\tim) + \sum_{\i=1}^{m}\lambda_\i\nabla\ell_\i(\th_\tim)$. 
Correspondingly, 
the $\Phi(\lambda)$ has the following dual form: 
\[
\max_{\lambda\in\RRplus^m}
-\frac{1}{2}\left\Vert \nabla F(\cc_\tim)+\sum_{\i=1}^{m}\lambda_\i\nabla\ell_\i(\th_\tim)\right\Vert ^{2}+\sum_{\i=1}^{m}\lambda_\i\phi_\tim.  %
\]
This concludes the proof. 
\end{proof}

\paragraph{Theorem~\ref{thm:odeell} [Pareto Improvement on $\L$]}
 \emph{Under Assumption~\ref{asm:basic}, 
assume $\cc_0\not\in\P_\ep$, and $t_\ep$ is the first time when $\cc_{t_\ep} \in \P_\ep$, then for any time $t < t_\ep,$ 
\bb 
\frac{\df}{\df t} \ell_\i( \cc_t) \leq -\alpha_t g(\cc_t),%
&&&
\min_{s\in[0,t]} g(\cc_t) \leq \frac{\min_{\i\in[m]}(\ell_\i(\cc_0)-\ell_i\true)}{\int_0^t \alpha_s \df s}.
\ee 
Therefore,  the update yields Pareto improvement on $\L$ when $\cc_t \not\in \P_\ep$ and  $\alpha_t g(\cc_t)>0$.\\
Further, 
if $\int_0^t \alpha_s \df s = +\infty$, 
then for any $\epsilon>\ep$,  there exists a finite time $t_\epsilon \in \RRplus$ on which the solution enters $\P_{\epsilon}$ and stays within $\overline \P_\epsilon$ afterwards, that is, we have $\cc_{t_\epsilon} \in \P_\epsilon$ and $\cc_t\in \overline \P_\epsilon$ 
for any $t \geq t_\epsilon$. 
}
\begin{proof} 

i) 
When $t < t_\ep$, we have $g(\cc_t) > \ep$ and hence 
\begin{align}\label{equ:tbl}
\frac{\df}{\df t} \ell_\i(\th_{\tim})
= - \nabla\ell_\i(\th_\tim)\tt v_\tim 
\leq - \phi_\tim = -\alpha_\tim g(\cc_\tim),
\end{align}
where we used the constraint of $\dd \ell_\i(\cc_\tim)\tt v_\tim\geq \phi_\tim$ in \eqref{opt: relax}. 
Therefore, we yield strict decent on all the losses $\{\ell_\i\}$
when $\alpha_t g(\cc_t) >0$. 

ii) 
Integrating both sides of \eqref{equ:tbl}: 
$$
\min_{s\in[0,t]}{g(\cc_s)} \leq \frac{\int_0^t \alpha_{s} g(\cc_s) \df s}{\int_0^t\alpha_s \df s} \leq 
\frac{\ell_\i(\th_0) - \ell_\i(\th_t)}{\int_0^t\alpha_s \df s} 
\leq \frac{\ell_\i(\th_0) - \ell\true}{\int_0^t\alpha_s \df s}.
$$
 This yields the result since it holds for every $i\in[m]$.
 
 If $\int_0^\infty \alpha_t \df t = +\infty$, then we have $\min_{s\in [0,t]} g(\cc_s) \to 0$ when $t\to +\infty$. Assume there exists an $\epsilon > \ep$, such that $\cc_t$ never enters $\P_\epsilon$ at finite $t$. Then we have $g(\cc_t) \geq \epsilon$ for $t\in \RRplus$,  which contradicts with  $\min_{s\in [0,t]} g(\cc_s) \to 0$.

iii) 
Assume there exists a finite time $t' \in (t_\epsilon, +\infty)$ such that $\cc_{t'}\not\in \overline \P_\epsilon$. 
Because $\epsilon > \ep$ and $g$ is continuous, 
$\P_e$ is in the interior of $\P_\epsilon \subseteq \overline \P_\epsilon$. 
Therefore, the trajectory leading to $\cc_{t'}\not\in \overline \P_\epsilon$  must pass through $\overline \P_{\epsilon}\setminus \P_\ep$ at some point, that is, there exists a point $t'' \in [t_\epsilon, t')$, such that $\{\cc_{t} \colon t\in[ t'', t'] \} \not\in  \P_\ep$. 
But because the algorithm can not increase any objective $\ell_i$ outside of $\P_\ep$, we must have $\L(\cc_{t'}) \preceq \L(\cc_{t''})$, yielding that $\cc_{t'}\in\overline{\{\cc_{t''}\}} \subseteq \overline \P_{\epsilon}$, where $\overline{\{\cc_{t''}\}}$ is the Pareto closure of ${\{\cc_{t''}\}}$; this contradicts with the assumption. 
\end{proof} 

\paragraph{Theorem~\ref{lem:dfjijgfgfgfgifjg}} 
\emph{ 
Under Assumption~\ref{asm:basic}, 
assume $\cc_t \not\in \P_\ep$  
is a fixed point of the algorithm, that is, $\frac{\df \theta_t}{\df t} = -v_t = 0$, 
and $F$, $\L$ are  convex in a neighborhood $\cc_t$, then 
$\cc_t$ is a local minimum of $F$
in the Pareto closure $\overline{\{\theta_t\}}$, %
that is,  there exists a neighborhood of $\cc_t$ in which  there exists no point $\cc'$ such that $F(\cc') < F(\cc_t)$ and $\L(\cc') \preceq \L(\cc_t)$. 
} 
\begin{proof} 
Note that minimizing $F$ in $\overline{\{\theta_t\}}$ can be framed into a constrained optimization problem:
$$
\min_{\cc} F(\cc) ~~~~s.t.~~~~\ell_i(\cc) \leq \ell_i(\cc_t),~~\forall i\in[m]. 
$$
In addition, by assumption, $\cc=\cc_t$ satisfies $v_t = \dd F(\cc_t) + \sum_{i=1}^m \lambda_{i,t}\dd \ell_i(\cc_t) = 0$, which is the KKT stationarity condition of the constrained optimization. 
It is also obvious to check that $\cc=\cc_t$ satisfies the feasibility and slack condition trivially. Combining this with the local convexity assumption yields the result. 
\end{proof}

\paragraph{Theorem~\ref{thm:odef} [Optimization of $F$]}
\emph{Let  $\epsilon > \ep$ and assume $g_{\epsilon} \defeq \sup_{\cc} \{g(\cc) \colon ~\cc\in \overline \P_\epsilon\}<+\infty$ and  $\sup_{t\geq0}\alpha_t<\infty$.   
Under Assumption~\ref{asm:basic}, 
when we initialize from $\cc_0 \in \P_\epsilon$, %
we have 
$$
\min_{s\in[0,t]}\norm{ \frac{\df\cc_s}{\df s} }^2 \leq 
\frac{F(\cc_0)-F\true}{t} + \frac{1}{t}\int_{0}^t 
\alpha_s \left (\alpha_s {g_\epsilon}    + 
c \sqrt{g_\epsilon}  \right ) 
\df s. 
$$
In particular, 
if we have $\alpha_t = \alpha = const$, then 
$\min_{s\in[0,t]}\norm{\df \theta_s/\df s}^2  = \bigO\left (1/t + \alpha \sqrt{g_\epsilon} \right ).$ \\
If  
$%
\int_0^\infty\alpha_t^\gamma\df t <+\infty$ for some $\gamma \geq 1$, we have 
$\min_{s\in[0,t]}\norm{\df \theta_s/\df s}^2  = \bigO(1/t + \sqrt{g_\epsilon}/t^{1/\gamma}).$
} 
\begin{proof} 
i) 
The slack condition of the constrained optimization in \eqref{opt: relax} says that 
\bbb \label{equ:slack001}
\lambda_{i,t}\left(\nabla\ell_\i(\th_\tim)\tt v_\tim
 -\phi_\tim \right)=0,\ \forall \i\in[m].
\eee
This gives that 
\begin{align}
\left\Vert v_\tim \right\Vert ^{2} 
& = \left ({\nabla F(\th_\tim)+\sum_{\i=1}^m \lambda_{i,t}\nabla\ell_\i(\th_\tim)} \right ) \tt v_\tim \notag \\ 
& = \dd F(\th_\tim) \tt v_\tim + \sum_{\i=1}^m \lambda_{i,t} \phi_\tim \ant{plugging \eqref{equ:slack001}}. \label{equ:ddfdifd1}
\end{align}
If $\cc_t\not\in \P_\ep$, we have $\phi_t = \alpha_t g(\cc_t)$ and this gives 
\bb 
\frac{\df }{\df t} F(\cc_t) 
= - \dd F(\cc_t)\tt v_t  = 
- \left\Vert v_\tim \right\Vert ^{2} + \sum_{\i=1}^m \lambda_{i,t} \phi_\tim 
= - \left\Vert \frac{\df \theta_\tim}{\df t} \right\Vert ^{2} + \sum_{\i=1}^m \lambda_{i,t} \alpha_t g(\cc_t)
\ee 
If $\cc_t$ is in the interior of $\P_\ep$, then we run typical gradient descent of $F$ and hence has 
\bb 
\frac{\df }{\df t} F(\cc_t)  = 
- \left\Vert v_\tim \right\Vert ^{2} = -  \left\Vert \frac{\df \theta_\tim}{\df t} \right\Vert ^{2}. 
\ee  
If $\cc_t$ is on the boundary of $\P_\ep$, then by the definition of differential inclusion, 
$\df \cc/\df t$ belongs to the convex hull of the velocities that it receives from either side of the boundary, yielding that 
$$
\frac{\df }{\df t} F(\cc_t)= 
- \left\Vert \frac{\df \theta_\tim}{\df t} \right\Vert ^{2} + \beta \sum_{\i=1}^m \lambda_{i,t} \alpha_t g(\cc_t) 
\leq - \left\Vert \frac{\df \theta_\tim}{\df t} \right\Vert ^{2} +  \sum_{\i=1}^m \lambda_{i,t} \alpha_t g(\cc_t),
$$
where $\beta \in[0,1]$. 
Combining all the cases gives
$$
\frac{\df }{\df t} F(\cc_t) \leq 
- \left\Vert \frac{\df \theta_\tim}{\df t} \right\Vert ^{2} +  \sum_{\i=1}^m \lambda_{i,t} \alpha_t g(\cc_t). %
$$
Integrating this yields
\bb
\min_{s\in[0,t]}\norm{ \frac{\df\cc_s}{\df s} }^2
\leq \frac{1}{t} \int_0^t \norm{ \frac{\df\cc_s}{\df s} }^2 \df s 
& \leq\frac{F(\cc_0)-F\true}{t} + \frac{1}{t}\int_{0}^t 
\sum_{\i=1}^m \lambda_{i,s} \alpha_s g(\cc_s) \df s \\ 
& \leq\frac{F(\cc_0)-F\true}{t} + \frac{1}{t}\int_{0}^t  \alpha_s \left (\alpha_s {g_\epsilon}    + c \sqrt{g_\epsilon}  \right ) \df s,
\ee  
where the last step used Lemma~\ref{lem:cgbound} with $\phi_t = \alpha_t g(\cc_t)$:
$$
 \sum_{\i=1}^m \lambda_{i,t} \alpha_\tim g(\cc_\tim) 
\leq  \alpha_\tim^2 g(\cc_\tim) + c \alpha_\tim \sqrt{g(\cc_\tim)} 
\leq  \alpha_\tim^2 g_\epsilon + c \alpha_\tim \sqrt{g_\epsilon},
$$
and here we used $g(\cc_\tim) \leq g_\epsilon$ because the trajectory is contained in $\overline \P_\epsilon$ following Theorem~\ref{thm:odeell}.

The remaining results follow Lemma~\ref{lem:alphako}. 
\end{proof}

\subsubsection{Technical Lemmas} 
\begin{lemma}\label{lem:cgbound}
Assume Assumption~\ref{asm:basic} holds. %
Define $g(\cc) = \min_{\omega\in \C^m} \norm{\sum_{\i=1}^m \omega_\i \dd \ell_\i(\cc)}^2$, where $\C^m$ is the probability simplex on $[m]$. 
Then for the $v_\tim$ and $\lambda_{i,t}$ defined in \eqref{opt: relax} and \eqref{equ: dual}, we have 
$$
\sum_{\i=1}^m \lambda_{i,t} g(\cc_\tim) \leq \max\left(\phi_\tim + c \sqrt{g(\cc_\tim)}, ~0\right). 
$$
\end{lemma} 
\begin{proof} 
The slack condition of the constrained optimization in \eqref{opt: relax} says that
\[\lambda_{i,t}\left(
\dd \ell_\i(\cc) \tt v_\tim -
\phi_\tim
\right)=0,~~~~\forall \i\in[m]. 
\]
Sum the equation over $\i\in[m]$ and 
note  
that $v_\tim = \dd F(\cc_\tim) + \sum_{\i=1}^m \lambda_{i,t} \dd \ell_\i(\cc_\tim)$. 
 We get 
\bbb \label{equ:eurd} 
\norm{\sum_{\i=1}^m \lambda_{i,t} \dd\ell_\i(\cc_\tim)}^2 + \left (\sum_{\i=1}^m \lambda_{i,t}  \dd \ell_\i(\cc_\tim) 
\right)\tt \dd F(\cc)  -\sum_{\i=1}^m \lambda_{i,t}  \phi_\tim  = 0. 
\eee 
Define 
\bb 
x_\tim= \norm{\sum_{\i=1}^m \lambda_{i,t} \dd\ell_\i(\cc_\tim)}^2,&&
\bar \lambda_\tim  = \sum_{\i=1}^m \lambda_{i,t}, &&
g_\tim=g(\cc_\tim) = \min_{\omega\in \C^m} \norm{\sum_{\i=1}^m\omega_\i \dd \ell_\i(\cc_\tim)}^2. 
\ee 
Then it is easy to see that $x_\tim \geq \bar \lambda_\tim^2 g_\tim$. 
Using Cauchy-Schwarz inequality, 
$$
\abs{\left (\sum_{\i=1}^m \lambda_{i,t}\dd \ell_\i(\cc)\right )\tt \dd F(\cc_\tim)}
\leq \norm{\dd F(\cc_\tim)}\norm{\sum_{\i=1}^m \lambda_{i,t}\dd \ell_\i(\cc)} \leq c \sqrt{x_\tim}, 
$$
where we used $\norm{\dd F(\cc_\tim)} \leq c$ by Assumption~\ref{asm:basic}. 
Combining this with \eqref{equ:eurd}, we have 
$$
\abs{x_\tim - \bar \lambda_\tim \phi_\tim} \leq c \sqrt{x_\tim}. 
$$
Applying Lemma~\ref{lem:xcg} yields the result. %
\end{proof}

\begin{lemma}\label{lem:xcg}
Assume $\phi \in \RR$, and 
$x,\lambda, c, g \in \RRplus$ are non-negative real numbers and they satisfy 
\bb 
\abs{x - \lambda \phi} \leq c \sqrt{x}, ~~~~~~~~ x\geq \lambda^2 g.
\ee 
Then we have $\lambda g \leq \max(0, \phi + c \sqrt{g}).$
$
$
\end{lemma} 
\begin{proof}
Square the first equation, we get 
$$f(x):= (x-\lambda \phi )^2 - c^2 x \leq 0,$$
where $f$ is a quadratic function. 
To ensure that $f(x) \leq 0$ has a solution that satisfies $ x \ge \lambda^2 g$, we need to have $f(\lambda^2 g) \leq 0$, that is,  
$$
f(\lambda^2g) = (\lambda^2g - \lambda \phi )^2 - c^2 \lambda ^2 g \leq 0.
$$
This can hold under two cases: 

Case 1: $\lambda =0;$

Case 2: $|\lambda g -\phi|\leq c \sqrt{g}$, and hence 
$\phi - c \sqrt{g} \leq \lambda g \leq \phi + c \sqrt{g}
$. 

Under both case, we have 
$$
\lambda g \leq \max(0, \phi + c \sqrt{g}). 
$$
\end{proof}

\begin{lemma}\label{lem:alphako}
Let $\{\alpha_\tim\colon t\in \RRplus\} \subseteq \RRplus$ 
be a non-negative sequence with 
$
A:=\left(\int_0^\infty \alpha_\tim^\gamma \df t \right)^{1/\gamma}  <\infty$, where $\gamma \geq 1$,  and $B = \sup_t \alpha_t <\infty$. 
Then we have 
$$
\frac{1}{t} \int_{0}^t \left (\alpha_s^2 +\alpha_s\right) \df s \leq  (B+1) A t^{-1/\gamma}. 
$$
\end{lemma}
\begin{proof}
Let $\eta = \frac{\gamma}{\gamma-1}$, so that $1/ \eta  + 1/\gamma= 1$. We have by Holder's inequality, 
\bb 
\int_{0}^t\alpha_s \df s 
\leq \left( \int_{0}^t \alpha_s^\gamma \df s\right )^{1/\gamma} \left (\int_{0}^t 1^\eta\df s \right )^{1/\eta} 
\leq A t^{1/\eta}= A t^{1-1/\gamma}. 
\ee 
and hence 
$$
\frac{1}{t}\int_{0}^t\left (\alpha_s^2  + \alpha_s\right )\df s  
\leq \frac{B+1}{t} \int_0^t\alpha_s \df s 
\leq  (B+1) A t^{-1/\gamma}. 
$$
\end{proof}

\clearpage

\section{Practical Implementation} \label{sec: practice}
\paragraph{Hyper-parameters} 
Our algorithm introduces two hyperparameters $\{\alpha_t\}$ and $\ep$ over vanilla gradient descent. We use constant sequence $\alpha_t=\alpha $ and we take $\alpha = 0.5$ unless otherwise specified. 
We choose $\ep$ by $\ep = \gamma   \ep_0$,  
where $\ep_0 $ is an exponentially discounted average of 
$\frac{1}{m}\sum_{\i=1}^{m}\left\Vert \nabla\ell_\i(\cc_t)\right\Vert^2$ over the trajectory so that it automatically scales with the magnitude of the gradients of the problem at hand. In the experiments of this paper, we simply fix $\gamma=0.1$ unless specified.

\paragraph{Solving the Dual Problem}
Our method requires to calculate 
 $\{\lambda_{i,t}\}_{t=1}^m$ with the dual optimization problem in \eqref{equ: dual}, which can be solved with any off-the-shelf convex quadratic programming tool. 
In this work, we use a very simple 
projected gradient descent 
to approximately solve \eqref{equ: dual}.  
We initialize $\{\lambda_{i,t}\}_{t=1}^m$ with a zero vector and terminate when the difference between the last two iterations is smaller than a threshold or the algorithm reaches the maximum number of iterations (we use 100 in all experiments). 

\section{Experiments}
\subsection{Finding Preferred Pareto Models}
\subsubsection{Ratio-based Criterion} \label{appendix_sec: ratio-based preference}
The non-uniformity score from \citep{mahapatra2020multi} that we 
use in Figure~\ref{fig: epo recover} is defined as 
\begin{align} \label{equ: nuf}
{F_\text{NU}(\ensuremath{\th})}=\sum_{t=1}^{m}\hat{\ell}_{t}(\th)\log\left (\frac{\hat{\ell}_{t}(\th)}{1/m}\right ),~~~~~\ \ \hat{\ell}_{t}(\th)=\frac{r_{t}\ell_{t}(\th)}{\sum_{s\in[m]}r_{s}\ell_{s}(\th)}. 
\end{align}

We fix the other experiment settings the same as \citet{mahapatra2020multi} and use $\gamma=0.01$ and $\alpha=0.25$ for this experiment reported in the main text. We defer the ablation studies on the hyper-parameter $\alpha$ and $\gamma$ to Section \ref{appendix_sec: dynamics}.

\subsubsection{ZDT2-Variant} \label{appendix_sec: zdt}
We consider the ZDT2-Variant example used in \citet{ma2020efficient} with the same experiment setting, in which the Pareto set is a cylindrical surface, making the problem more challenging. We consider the same criteria, e.g. weighted distance and complex cosine used in the main context with different choices of $r_1 = [0.2, 0.4, 0.6, 0.8]$. We use the default hyper-parameter set up, choosing $\alpha=0.5$ and $r=0.1$. For complex cosine, we use MGD updating for the first 150 iterations. Figure \ref{fig: zdt_general} shows the trajectories, demonstrating that PNG works pretty well for the more challenging ZDT2-Variant tasks.

\begin{figure}
\begin{centering}
\includegraphics[scale=0.4]{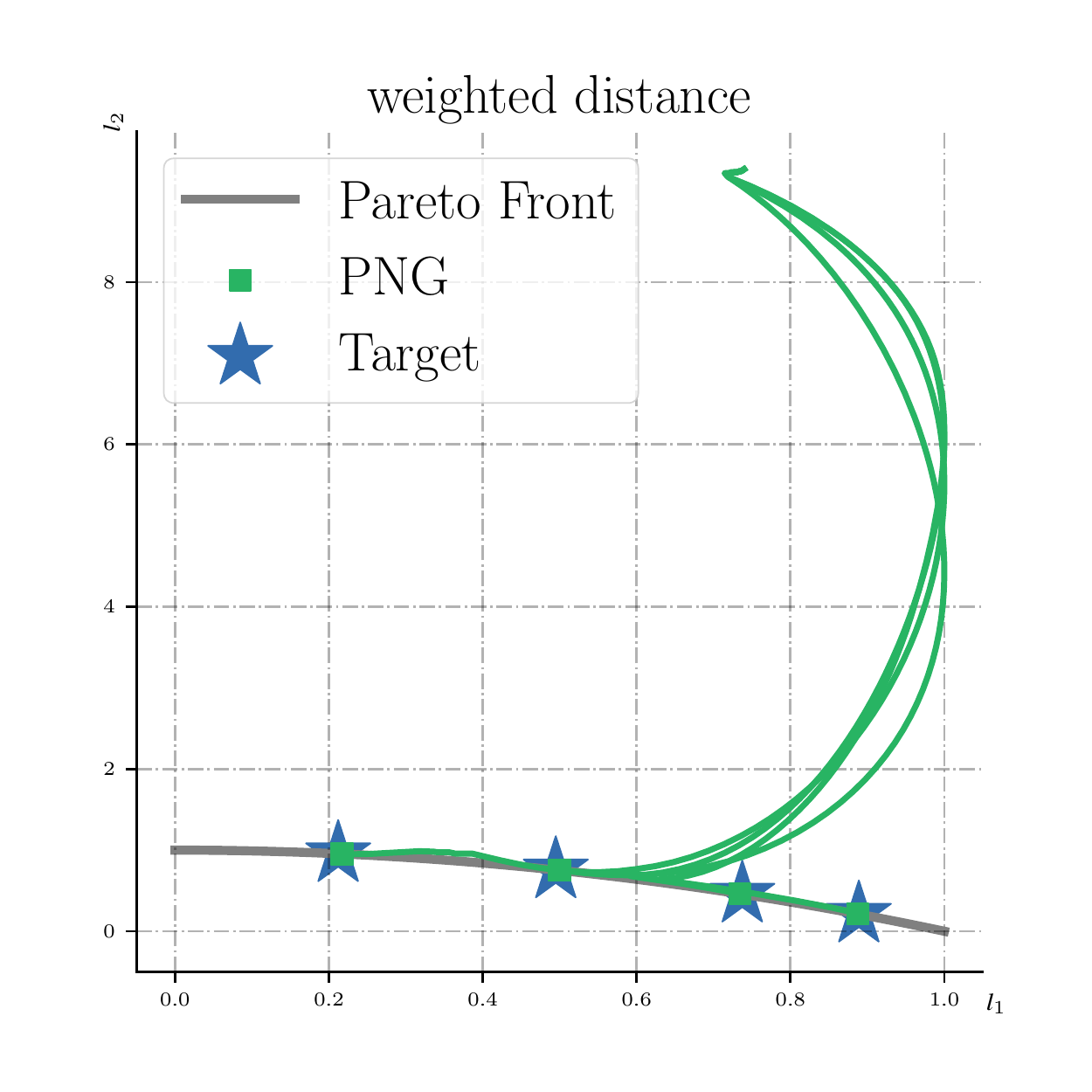}
\includegraphics[scale=0.4]{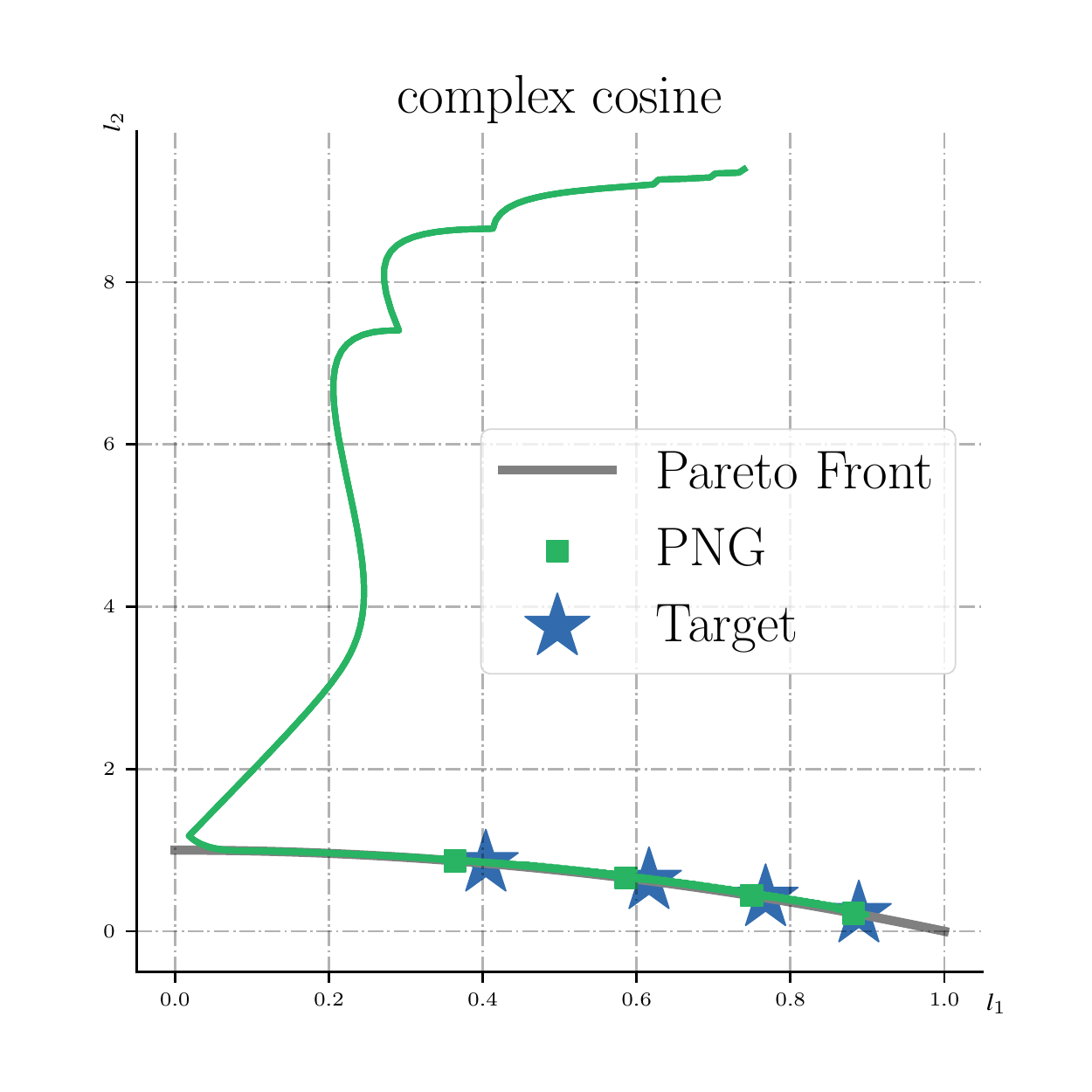}
\caption{Trajectories of solving OPT-in-Pareto with weighted distance and complex cosine as criterion using PNG. The green dots are the final converged models. PNG is able to successfully locate the correct models in the Pareto set.} \label{fig: zdt_general}
\end{centering}
\end{figure}

\subsubsection{General Criteria: Three-task learning on the NYUv2 Dataset} \label{appendix_sec: mtan}
We show that PNG is able to handle large-scale multitask learning problems by deploying it on a three-task learning problem (segmentation, depth estimation, and surface normal prediction) on NYUv2 dataset \citep{silberman2012indoor}. The main goal of this experiment is to show that: 1. PNG is able to handle OPT-in-Pareto in a large-scale neural network; 2. With a proper design of criteria, PNG enables to do targeted fine-tuning that pushes the model to move towards a certain direction. We consider the same training protocol as \citet{liu2019end} and use the MTAN network architecture. Start with a model trained with equally weighted linear scalarization and our goal is to further improve the model’s performance on segmentation and surface normal estimation while allowing some sacrifice on depth estimation. This can be achieved by many different choices of criterion and in this experiment, we consider the following design: 
$
F(\cc)= (\ell_{\text{seg}}(\cc) \times \ell_{\text{surface}}(\cc)) / (0.001 + \ell_{\text{depth}}(\cc)).
$
Here $\ell_{\text{seg}}$, $\ell_{\text{surface}}$ and $\ell_{\text{depth}}$ are the loss functions for segmentation, surface normal prediction and depth estimation, respectively. The constant 0.001 in the denominator is for numeric stability. We point out that our design of criterion is a simple heuristic and might not be an optimal choice and the key question we study here is to verify the functionality of the proposed PNG. As suggested by the open-source repository of \citet{liu2019end}, we reproduce the result based on the provided configuration. To show that PNG is able to move the model along the Pareto front, we show the evolution of the criterion function and the norm of the MGD gradient during the training in Figure \ref{fig: mtan_traj}. As we can see, PNG effectively decreases the value of criterion function while the norm of MGD gradient remains the same. This demonstrates that PNG is able to minimize the criterion by searching the model in the Pareto set. Table \ref{tbl: mtan} compares the performances on the three tasks using standard training and PNG, showing that PNG is able to improve the model’s performance on segmentation and surface normal prediction tasks while satisfying a bit of the performance in depth estimation based on the criterion.

\begin{figure}
\begin{centering}
\includegraphics[scale=0.4]{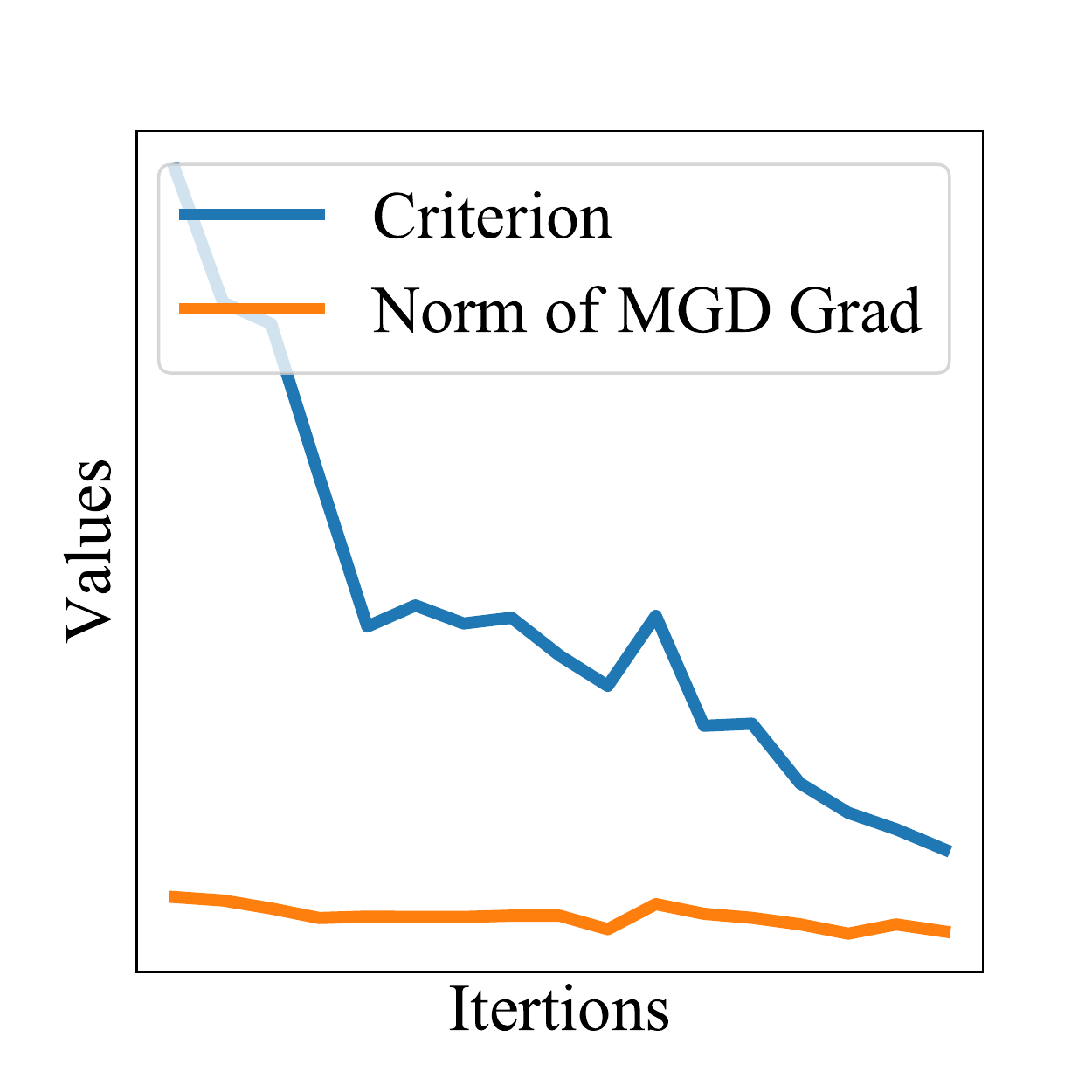}
\par\end{centering}
\caption{The evolution of Criterion $F$ and the norm of MGD gradient when trained using PNG on NYUv2 dataset with MTAN network. PNG effectively decreases the criterion while ensuring the model is within the Pareto set, since the norm of MGD gradient remains unchanged.} \label{fig: mtan_traj}
\end{figure}

\subsection{Finding Diverse Pareto Models} \label{appendix_sec: pareto approximation}
\subsubsection{Experiment Details} \label{appendix_sec: pareto approximation exp}
We train the model for 100 epochs using Adam optimizer with batch size 256 and 0.001 learning rate. To encourage diversity of the models, following the setting in \citet{mahapatra2020multi}, we use equally distributed preference vectors for linear scalarization and EPO. Note that the stochasticity of using mini-batches is able to improve the performance of Pareto approximation for free by also using the intermediate checkpoints to approximate $\P$. To fully exploit this advantage, for all the methods, we collect checkpoints every epoch to approximate $\P$, starting from epoch 60.
\subsubsection{Evaluation Metric Details} \label{appendix_sec: pareto approximation metric}
We introduce the definition of the used metric for evaluation. Given a set $\hat{\P} = \{\theta_1,\ldots, \cc_N\}$ that we use to approximate  $\P$, its IGD+ score is defined as: 
\[
\text{IGD$_+$}(\hat{\P})=\int_{\P^*}
q(\th,\hat{\P})d\mu(\th),\ \ \ \  q(\th,\hat{\P})=\min_{\hat{\th}\in\hat{\P}}\left\Vert \left(\L(\hat{\th})-\L(\th)\right)_{+}\right\Vert,
\]
where $\mu$ is some base measure that measures the importance of
$\th\in \P$ and $(t)_+\defeq \max(t,0)$, applied on each element of a vector. Intuitively, for each $\th$, we find a
nearest $\hat{\th}\in\hat{\P}$ that approximates $\th$ best. Here
the $(\cdot)_{+}$ is applied as we only care the tasks that $\hat{\th}$
is worse than $\th$. In practice, a common choice of $\mu$ can be
a uniform counting measure with uniformly sampled (or selected) models
from $\P$. In our experiments, since we can not sample models from $\P$, we approximate $\P$ by combining $\hat{\P}$ from all the methods, {\color{black} i.e., $\P\approx\cup_{m\in\text{\{Linear,MGD,EPO,PNG\}}}\hat{\P}_{m}$, where $\hat{P}_m$ is the approximation set produced by algorithm $m$.}

This approximation might not be accurate but is sufficient to compare the different methods,

The Hypervolume score of $\hat \P$, w.r.t. a reference point $\L^r\in \RRplus^m$,  is defined as 
\[
\text{HV}(\hat{\P}) = \mu\left(\left\{ \L=[\ell_{1},...,\ell_{m}]\in\mathbb{R}^{m}\mid\exists\th\in\hat{\P},\ \text{s.t.}\ \ell_{t}(\th)\le\ell_{t}\le\ell_{t}^{r}\ \forall t\in[m]\right\} \right),
\]
where $\mu$ is again some measure. We use $\L^r=[0.6, 0.6]$ for calculating the Hypervolume based on loss and set $\mu$ to be the common Lebesgue measure. Here we choose 0.6 as we observe that the losses of the two tasks are higher than 0.6 and 0.6 is roughly the worst case. When calculating Hypervolume based on accuracy, we simply flip the sign.

\begin{table}
\begin{centering}
\begin{tabular}{c|ccccccccc}
\toprule
\multirow{3}{*}{\makecell[c]{Algorithm}} & \multicolumn{2}{c}{Segmentation} & \multicolumn{2}{c}{Depth} & \multicolumn{5}{c}{Surface Normal}\tabularnewline
\cline{2-10} \cline{3-10} \cline{4-10} \cline{5-10} \cline{6-10} \cline{7-10} \cline{8-10} \cline{9-10} \cline{10-10} 
 & \multicolumn{2}{c}{(Higher Better)} & \multicolumn{2}{c}{(Lower Better)} & \multicolumn{2}{c}{\makecell[c]{Angle Distance \\ (Lower Better)}} & \multicolumn{3}{c}{Within $t^{\circ}$}\tabularnewline
 & mIoU & Pix Acc & Abs Err & Rel Err & Mean & Median & 11.25 & 22.5 & 30\tabularnewline
\hline 
Standard & 27.09 & 56.36 & 0.6143 & 0.2618 & 31.46 & 27.37 & 19.51 & 41.71 & 54.61\tabularnewline
PNG & 28.23 & 56.66 & 0.6161 & 0.2632 & 31.06 & 26.50 & 21.06 & 43.41 & 55.93\tabularnewline
\bottomrule
\end{tabular}\caption{Comparing the multitask performance of standard training using linear scalarization with equally weighted losses and the targeted fine-tuning based on PNG.} \label{tbl: mtan}
\par\end{centering}
\end{table}

\subsubsection{Ablation Study} \label{appendix_sec: pareto approximation abl}
We conduct ablation study to understand the effect of $\alpha$ and $\gamma$ using the Pareto approximation task on Multi-Mnist. We compare PNG with $\alpha=0.25, 0.5, 0.75$ and $\gamma=0.01, 0.1, 0.25$. Figure \ref{table: ablation} summarizes the result. Overall, we observe that PNG is not sensitive to the choice of hyper-parameter.

\begin{table}
\begin{centering}
\begin{tabular}{c|c|cccc}
\toprule
\multicolumn{2}{c|}{} & \multicolumn{2}{c}{Loss} & \multicolumn{2}{c}{Acc}\tabularnewline
\multicolumn{1}{c}{} &  & Hv$\uparrow$ ($10^{-2}$) & IGD$\downarrow$ ($10^{-2}$) & Hv$\uparrow$ ($10^{-2}$) & IGD$\downarrow$ ($10^{-2}$)\tabularnewline
\hline 
\multirow{3}{*}{$\gamma=0.1$} & $\alpha=0.25$ & $7.89\pm0.11$ & $0.041\pm0.012$ & $9.39\pm0.038$ & $0.0056\pm0.002$\tabularnewline
 & $\alpha=0.5$ & $7.86\pm0.12$ & $0.043\pm0.012$ & $9.39\pm0.038$ & $0.0056\pm0.002$\tabularnewline
 & $\alpha=0.75$ & $7.84\pm0.11$ & $0.045\pm0.013$ & $9.38\pm0.037$ & $0.0057\pm0.002$\tabularnewline
\hline 
\multirow{3}{*}{$\alpha=0.5$} & $\gamma=0.01$ & $7.86\pm0.12$ & $0.042\pm0.012$ & $9.39\pm0.038$ & $0.0056\pm0.002$\tabularnewline
 & $\gamma=0.1$ & $7.86\pm0.12$ & $0.043\pm0.012$ & $9.39\pm0.038$ & $0.0056\pm0.002$\tabularnewline
 & $\gamma=0.25$ & $7.85\pm0.11$ & $0.042\pm0.012$ & $9.39\pm0.036$ & $0.0056\pm0.002$\tabularnewline
\bottomrule
\end{tabular}
\par\end{centering}
\caption{Ablation study based on Multi-Mnist dataset with different choice of $\alpha$ and $\gamma$.} \label{table: ablation}
\end{table}

\subsubsection{Comparing with the Second Order Approach} \label{appendix_sec: compare second order}
We give a discussion on comparing our approach with the second order approaches proposed by \citet{ma2020efficient}. In terms of algorithm, \citet{ma2020efficient} is a local expansion approach. To apply \citet{ma2020efficient}, in the first stage, we need to start with several well distributed models (i.e., the ones obtained by linear scalarization with different preference weights) and \citet{ma2020efficient} is only applied in the second stage to find the neighborhood of each model. The performance gain comes from the local neighbor search of each model (i.e. the second stage).

In comparison, PNG with energy distance is a global search approach. It improves the well-distributedness of models in the first stage (i.e. it’s a better approach than simply using linear scalarization with different weights). And thus the performance gain comes from the first stage. Notice that we can also apply \citet{ma2020efficient} to PNG with energy distance to add extra local search to further improve the approximation.

In terms of run time comparison. We compare the wall clock run time of each step of updating the 5 models using PNG and the second order approach in \citet{ma2020efficient}. We calculate the run time based on the multi-MNIST dataset using the average of 100 steps. PNG uses 0.3s for each step while \citet{ma2020efficient} uses 16.8s. PNG is \emph{56x} faster than the second order approach. And we further argue that, based on time complexity theory, the gap will be even larger when the size of the network increases.

\subsection{Trajectory Visualization with Different Hyper-parameters} \label{appendix_sec: dynamics}
We give visualization on the PNG trajectory when using different hyper-parameters. We reuse synthetic example introduced in Section \ref{sec: subset application} for studying the hyper-parameters $\alpha$ and $\gamma$. We fix $\alpha=0.25$ and vary $\gamma = 0.1, 0.05, 0.01, 0.1$; and fix $\gamma=0.01$ and vary $\alpha=0.1, 0.25, 0.5, 0.75$. Figure \ref{fig: epo recover ablation} plots the trajectories. As we can see, when $\gamma$ is properly chosen, with different $\alpha$, PNG finds the correct models with different trajectories. Different $\alpha$ determines the algorithm's behavior of balancing the descent of task losses or criterion objectives. On the other hand, with too large $\gamma$, the algorithm fails to find a model that is close to $\P^*$, which is expected.

\begin{figure}
\begin{centering}
\includegraphics[scale=0.32]{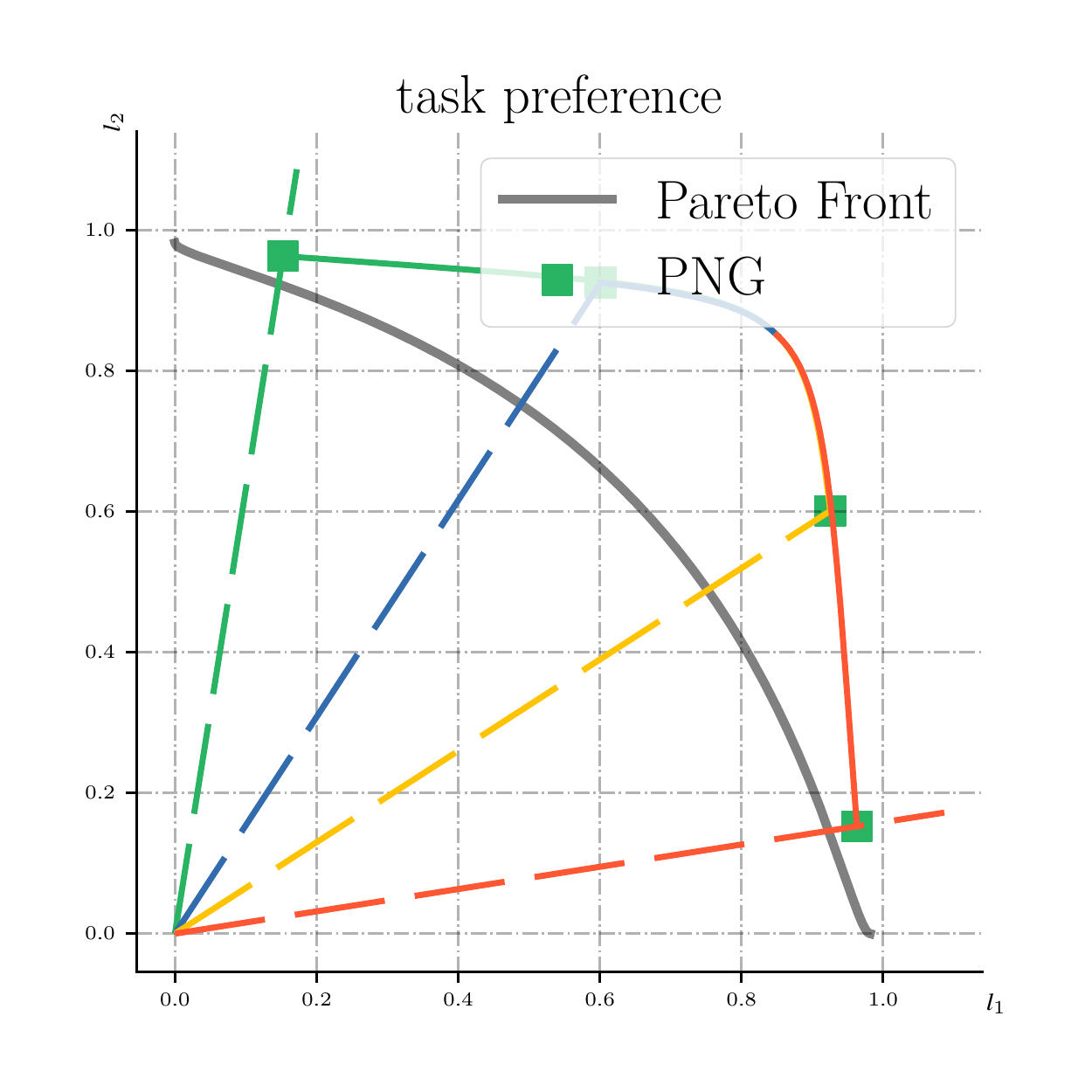}\hspace{-0.6cm}
\includegraphics[scale=0.32]{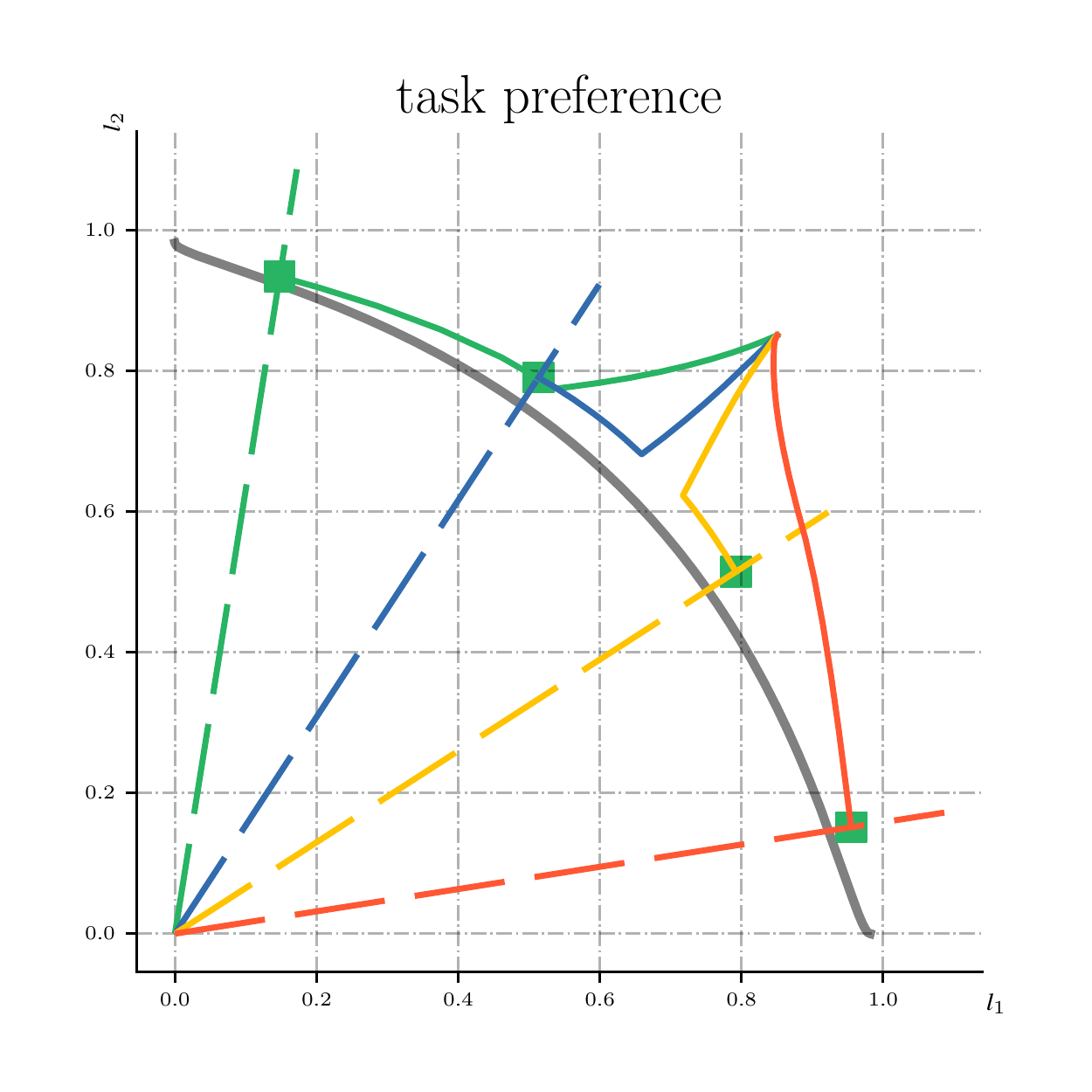}\hspace{-0.6cm}
\includegraphics[scale=0.32]{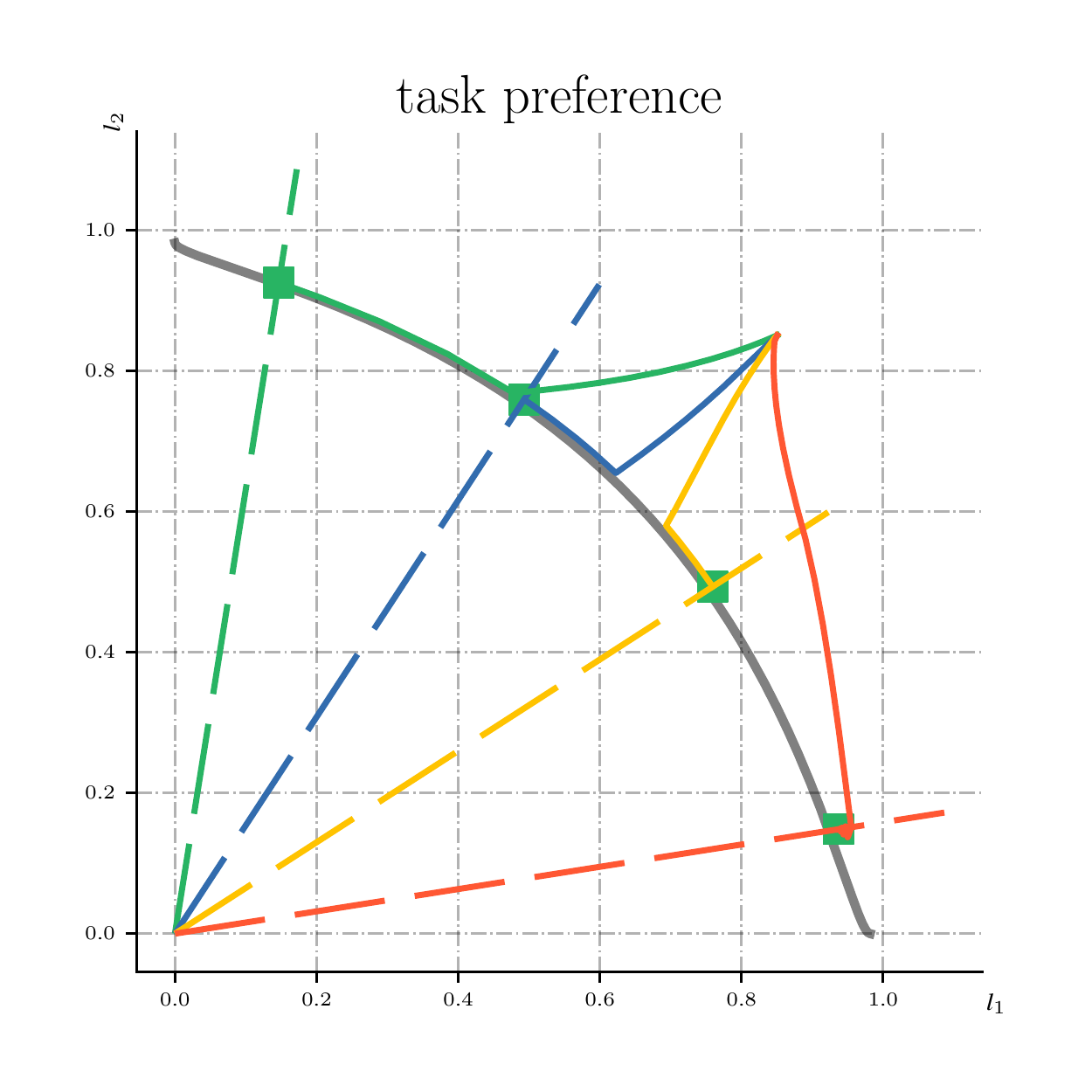}\hspace{-0.6cm}
\includegraphics[scale=0.32]{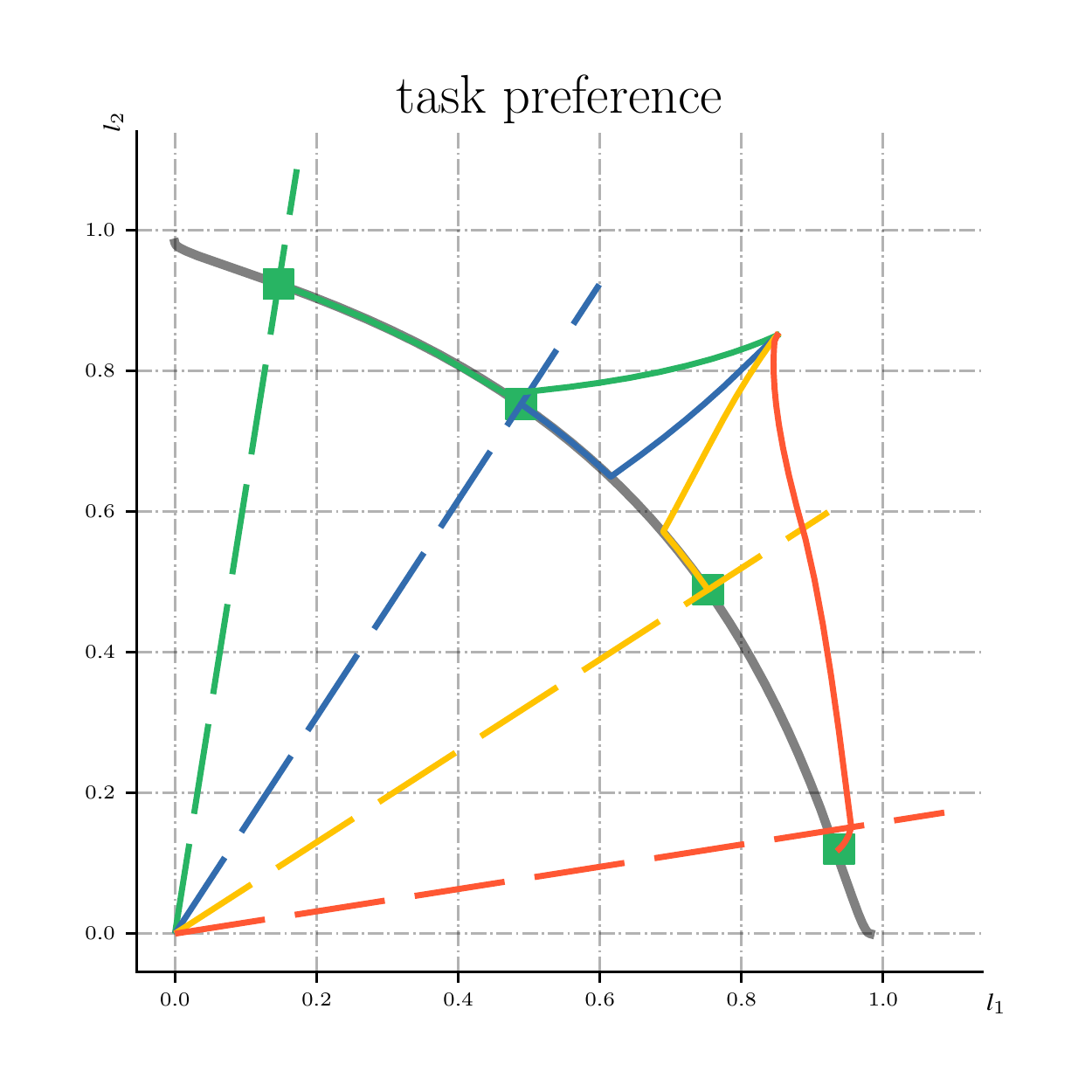}

\includegraphics[scale=0.32]{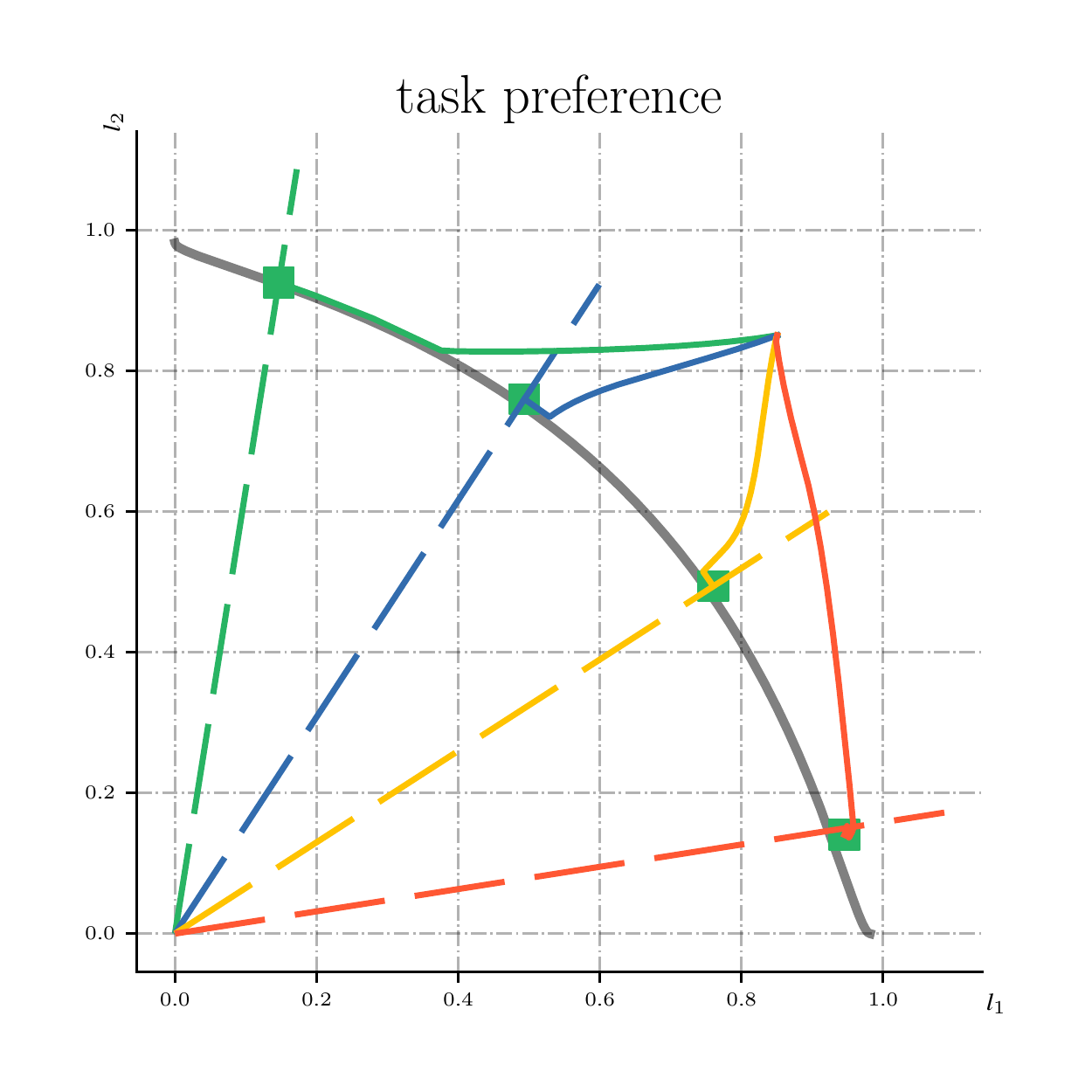}\hspace{-0.6cm}
\includegraphics[scale=0.32]{fig/ablation/NPOalpha_0.25thre_0.01_recover.pdf}\hspace{-0.6cm}
\includegraphics[scale=0.32]{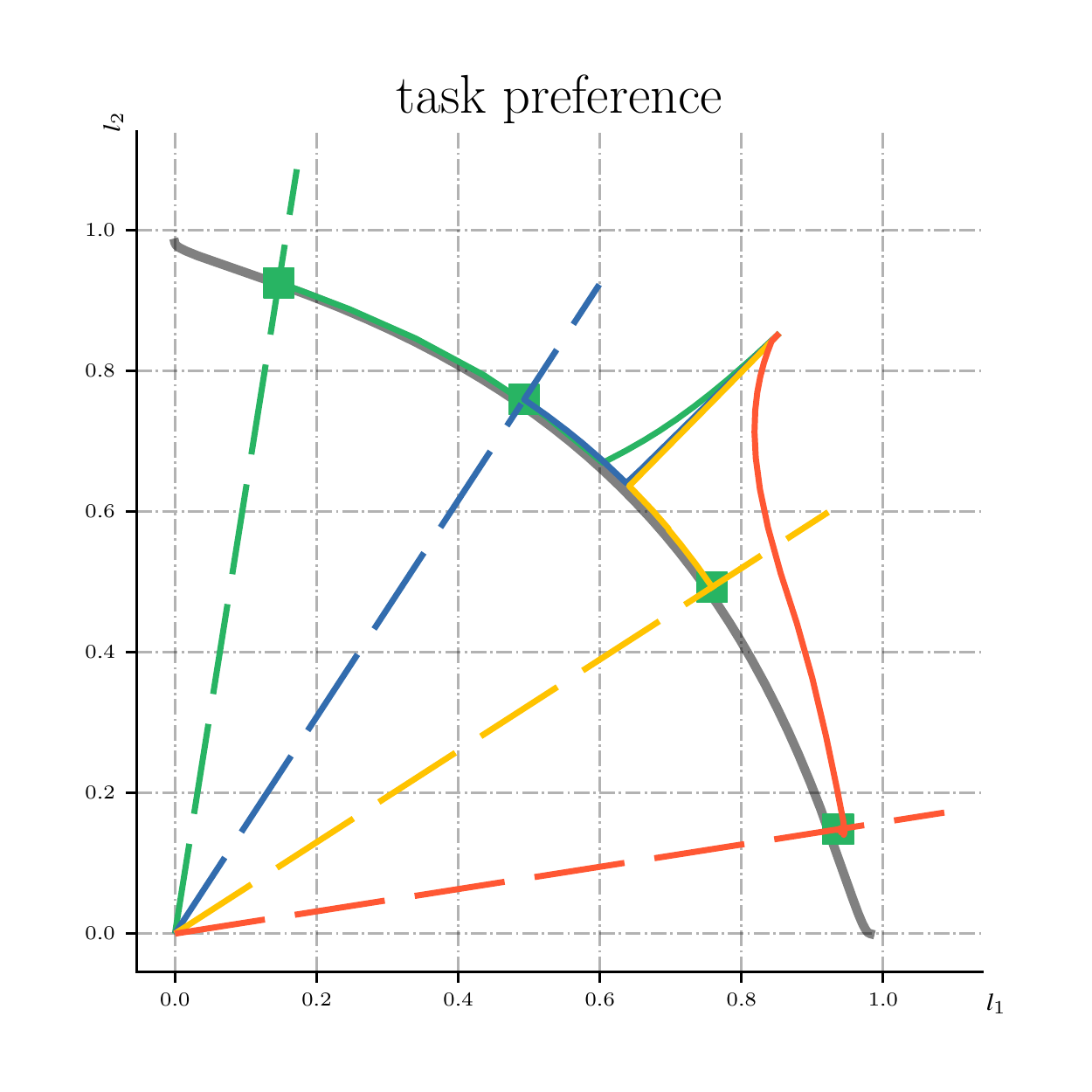}\hspace{-0.6cm}
\includegraphics[scale=0.32]{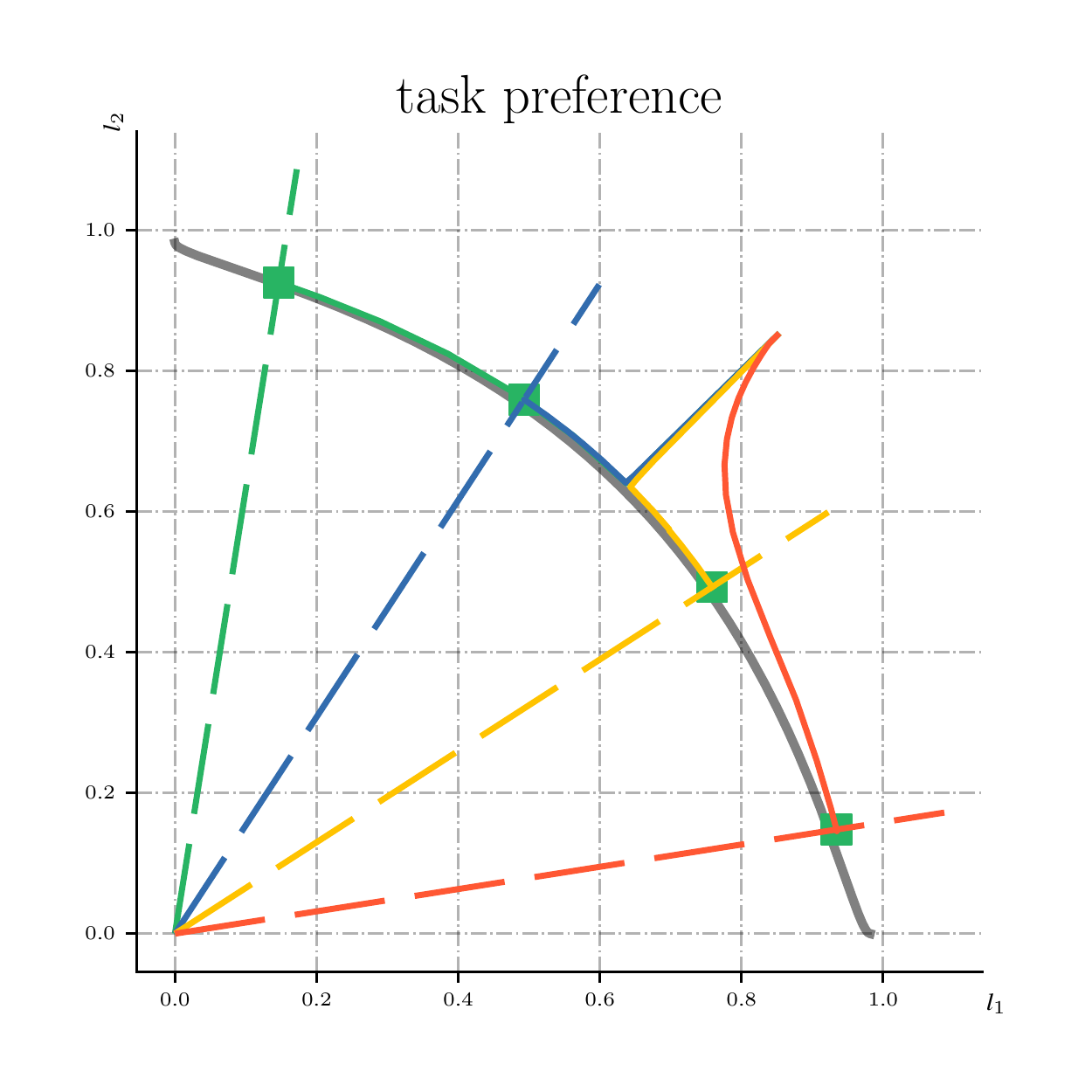}
\par\end{centering}
\caption{Ablation study on OPT-in-Pareto with different ratio constraint of objectives. Upper row, from left to right: fixing $\alpha=0.25$, $\gamma=0.1, 0.05, 0.01, 0.001$; Lower row, from left to right: fixing $\gamma=0.01$, $\alpha=0.1, 0.25, 0.5, 0.75$. By comparing the figures in the first row, we find that choosing a too large $\gamma$ make the final converged model be far away from the Pareto set, which is as expected. By comparing the figures in the second row, we find that changing $\alpha$ make PNG give different priority in making Pareto improvement or descent on $F$. When $\alpha$ is larger (the right figures), PNG will first move the model to Pareto set and start to decrease $F$ after that.} \label{fig: epo recover ablation}
\end{figure}

\subsection{Improving Multitask Based Domain Generalization} \label{appendix_sec: dg}

We argue that many other deep learning problems also have the structure
of multitask learning when multiple losses presents and thus optimization
techniques in multitask learning can also be applied to those domains. In this paper we consider the JiGen \citep{carlucci2019domain}.
JiGen learns
a model that can be generalized to unseen domain by minimizing a standard
cross-entropy loss $\ell_{\text{class}}$ for classification and an
unsupervised loss $\ell_{\text{jig}}$ based on Jigsaw Puzzles: 
\[
\ell(\th)=(1-\omega)\ell_{\text{class}}(\th)+\omega\ell_{\text{jig}}(\th).
\]
The ratio between two losses, i.e. $\omega$, is important to the final performance of the model and requires a careful grid search. Notice that JiGen is essentially searching for a model on the Pareto front using the linear scalarization. Instead of using a fixed linear scalarization to learn a model, one natural questions is that whether it is possible to design a mechanism that dynamically adjusts the ratio of the losses so that we can achieve to learn a better model.

We give a case study here. Motivated by the
adversarial feature learning \citep{JMLR:v17:15-239}, we propose to improve JiGen such that the latent feature representations of the two tasks are well aligned. Specifically, suppose that {\color{black}$\Phi_{\text{class}}(\th)=\{\phi_{\text{class}}(x_{i},\th)\}_{i=1}^{n}$ and $\Phi_{\text{jig}}(\th)=\{\phi_{\text{jig}}(x_{i},\th)\}_{i=1}^{n}$
is the distribution of latent feature representation of the two tasks, where $x_i$ is the $i$-th training data.}
We consider $F_{\text{PD}}$ as some probability metric that measures
the distance between two distributions, we consider the following
problem: 
{\color{black}
\[
\min_{\th\in \P^*}F_{\text{PD}}[\Phi_{\text{class}}(\th),\Phi_{\text{jig}}(\th)].
\]
}
With $\text{PD}$ as the criterion function, our algorithm automatically reweights the ratio of the two tasks such that their latent space is well aligned.

\begin{table}[t]
\begin{centering}
\scalebox{0.92}{
\begin{tabular}{c|cccc|c}
\toprule
Method & Art paint & Cartoon & Sketches & Photo & Avg\tabularnewline
\hline 
\multicolumn{6}{c}{AlexNet}\tabularnewline
\hline 
TF & $0.6268$ & $0.6697$ & $0.5751$ & $0.8950$ & $0.6921$\tabularnewline
CIDDG & $0.6270$ & $0.6973$ & $0.6445$ & $0.7865$ & $0.6888$\tabularnewline
MLDG & $0.6623$ & $0.6688$ & $0.5896$ & $0.8800$ & $0.7001$\tabularnewline
D-SAM & $0.6387$ & $0.7070$ & $0.6466$ & $0.8555$ & $0.7120$\tabularnewline
DeepAll & $0.6668$ & $0.6941$ & $0.6002$ & $0.8998$ & $0.7152$\tabularnewline
\hline 
JiGen & $0.6855\pm0.004$ & $\pmb{0.6889\ensuremath{\pm}0.002}$ & $\pmb{0.6831\ensuremath{\pm}0.011}$ & $0.8946\pm0.008$ & $0.7380\pm0.002$\tabularnewline
JiGen + adv & $0.6857\pm0.004$ & $0.6837\pm0.003$ & $0.6753\pm0.008$ & $0.8980\pm0.001$ & $0.7357\pm0.003$\tabularnewline
Jigen + PNG & $\pmb{0.6914\ensuremath{\pm}0.005}$ & $\pmb{0.6903\ensuremath{\pm}0.002}$ & $\pmb{0.6855\ensuremath{\pm}0.007}$ & $\pmb{0.9044\ensuremath{\pm}0.003}$ & $\pmb{0.7429\ensuremath{\pm}0.002}$\tabularnewline
\hline 
\multicolumn{6}{c}{ResNet-18}\tabularnewline
\hline 
D-SAM & $0.7733$ & $0.7243$ & $0.7783$ & $0.9530$ & $0.8072$\tabularnewline
DeepAll & $0.7785$ & $0.7486$ & $0.6774$ & $0.9573$ & $0.7905$\tabularnewline
\hline 
JiGen & $0.8009\pm0.004$ & $0.7363\pm0.007$ & $0.7046\pm0.013$ & $\pmb{0.9629\ensuremath{\pm}0.002}$ & $0.8012\pm0.002$\tabularnewline
JiGen + adv & $0.7923\pm0.006$ & $0.7402\pm0.004$ & $0.7188\pm0.005$ & $0.9617\pm0.001$ & $0.8033\pm0.001$\tabularnewline
JiGen + PNG & $\pmb{0.8014\ensuremath{\pm}0.005}$ & $\pmb{0.7538\ensuremath{\pm}0.001}$ & $\pmb{0.7222\ensuremath{\pm}0.006}$ & $\pmb{0.9627\ensuremath{\pm}0.002}$ & $\pmb{0.8100\ensuremath{\pm}0.005}$\tabularnewline
\bottomrule
\end{tabular}
}
\par\end{centering}
\caption{Comparing different algorithms for domain generalization using dataset PACS and two network architectures. The setting is the same to that of Table \ref{tbl: domain_small}.} \label{tbl: dg}
\end{table}

\textbf{Setup} We fix all the experiment setting the same as \citet{carlucci2019domain}. We use the Alexnet and Resnet-18 with multihead pretrained on ImageNet as the multitask network. We evaluate the methods on PACS \citep{Li_2017_ICCV}, which covers 7 object categories and 4 domains (Photo, Art Paintings, Cartoon and Sketches). Same to \citet{carlucci2019domain}, we trained our model considering three domains as source datasets and the remaining one as target. We implement $F_\text{PD}$ that measures the discrepancy of the feature space of the two tasks using the idea of Domain Adversarial Neural Networks \citep{ganin2015unsupervised} by adding an extra prediction head on the shared feature space to predict the whether the input is for the classification task or Jigsaw task. {\color{black}Specifically, we add an extra linear layer on the shared latent feature representations that is trained to predict the task that the latent space belongs to, i.e.,
\[
F_{\text{PD}}(\Phi_{\text{class}}(\th),\Phi_{\text{jig}}(\th))=\min_{w,b}\frac{1}{n}\sum_{i=1}^{n}\log(\sigma(w^{\top}\phi_{\text{class}}(x_{i},\th)))+\log(1-\sigma(w^{\top}\phi_{\text{class}}(x_{i},\th))).
\]
Notice that the optimal weight and bias for the linear layer depends on the model parameter $\th$, during the training, both $w,b$ and $\th$ are jointly updated using stochastic gradient descent. We follow the default training protocol provided by the source code of \citet{carlucci2019domain}. 
}

\textbf{Baselines} Our main baselines are JiGen \citep{carlucci2019domain}; JiGen + adv, which adds an extra domain adversarial loss on JiGen; and our PNG with domain adversarial loss as criterion function. In order to run statistical test for comparing the methods, we run all the main baselines using 3 random trials. We use the released source code by \citet{carlucci2019domain} to obtained the performance of JiGen. For JiGen+adv, we use an extra run to tune the weight for the domain adversarial loss. Besides the main baselines, we also includes TF \citep{Li_2017_ICCV}, CIDDG \citep{li2018deep}, MLDG \citep{li2018learning} , D-SAM \citep{d2018domain} and DeepAll \citep{carlucci2019domain} as baselines with the author reported performance for reference.

\textbf{Result} The result is summarized in Table \ref{tbl: dg} with bolded value indicating the statistical significant best methods with p-value based on matched-pair t-test less than 0.1. Combining Jigen and PNG to dynamically reweight the task weights is able to implicitly regularizes the latent space without adding an actual regularizer which might hurt the performance on the tasks and thus improves the overall result.

\end{document}